\documentclass[reqno]{amsart}
\usepackage{amsmath,amsfonts,amssymb,amsthm,enumitem}
\usepackage[usenames,dvipsnames]{xcolor}
\usepackage[colorlinks=true,linkcolor=blue]{hyperref} %para PFDLaTeX
\newtheorem{theorem}{Theorem}[section]
\newtheorem{lemma}[theorem]{Lemma}
\newtheorem{proposition}[theorem]{Proposition}

\theoremstyle{definition}
\newtheorem{definition}[theorem]{Definition}

\newtheorem{remark}[theorem]{Remark}

\numberwithin{equation}{section}
 \def\N{\mathbb N}  
\def\R{\mathbb R}   \def\eps{\varepsilon}
 
\newcommand{\w}{\omega} \newcommand{\Om}{\Omega}
\newcommand{\di}{\textsf{d}}
\newcommand{\wh}{\widehat}
\newcommand{\wt}{\widetilde} \newcommand{\n}[1]{\| #1 \|}
\begin{document}
\title[Asymptotic behavior of neutral dynamical systems]{Asymptotic behavior of solutions of nonautonomous neutral dynamical systems}
\author[S.~Novo]{Sylvia Novo} \author[R.~Obaya]{Rafael Obaya}
\author[V.M.~Villarragut]{V\'{\i}ctor M. Villarragut} \address[Sylvia Novo,
Rafael Obaya]{Departamento de Matem\'{a}tica Aplicada, Universidad de
  Valladolid, Paseo del Cauce 59, 47011 Valladolid, Spain.}
\address[V\'{\i}ctor~M.~Villarragut]{Departamento de Matem\'{a}tica Aplicada a
  la Ingenier\'{\i}a Industrial, Universidad Polit\'{e}cnica de Madrid, Calle de
  Jos\'{e} Guti\'{e}rrez Abascal 2, 28006 Madrid, Spain.}  \thanks{The first two
  authors are partly supported by MICIIN/FEDER under project
  RTI2018-096523-B-100 and EU Marie-Sk\l odowska-Curie ITN Critical Transitions
  in Complex Systems (H2020-MSCA-ITN-2014 643073 CRITICS)} \email[Sylvia
Novo]{sylnov@wmatem.eis.uva.es} \email[Rafael Obaya]{rafoba@wmatem.eis.uva.es}
\email[V\'{\i}ctor M. Villarragut]{victor.munoz@upm.es}
\date{}
\begin{abstract}
  This paper studies the dynamics of families of monotone nonautonomous neutral
  functional differential equations with nonautonomous operator, of great importance for their
  applications to the study of the long-term behavior of the trajectories of
  problems described by this kind of equations, such us compartmental systems and neural networks among many others. Precisely, more general
  admissible initial conditions are included in the study to show that the
  solutions are asymptotically of the same type as the coefficients  of the neutral and non-neutral part.
\end{abstract}
\maketitle
%%%%%%%%%%%%%%%%%%%%%%%%%%%%%%%%%%%%%%%%%%%%%%%%%%%%%%%%%%%%%%%%%%%%%%%%%%%%%%%%%%%%%
%%%%%%%%%%%%%%%%%%%%%%%%%%%%%%%%%%%%%%%%%%%%%%%%%%%%%%%%%%%%%%%%%%%%%%%%%%%%%%%%%%%%%
%%%%%%%%%%%%%%%%%%%%%%%%%%%%%%%%%%%%%%%%%%%%%%%%%%%%%%%%%%%%%%%%%%%%%%%%%%%%%%%%%%%%%
%%%%%%%%%%%%%%%%%%%%%%%%%%%%%%%%%%%%%%%%%%%%%%%%%%%%%%%%%%%%%%%%%%%%%%%%%%%%%%%%%%%%%
\section{Introduction}\label{secintro}
This paper studies the long-term behavior of the trajectories of a monotone
skew-product semiflow,
$\tau: \R^+ \times \Omega \times X \rightarrow \Omega \times X$,
$(t, \omega,x) \mapsto (\omega{\cdot} t, u(t,\omega,x))$, generated by a family
of nonautonomous differential equations. The base of the phase space, $\Omega$,
is a compact metric space endowed with a global recurrent flow
$(\Omega,\sigma,\R)$, and the fiber, $X$, is a Banach space with a positive cone
$X_+$ that induces an order relation. The monotone character means that, if
$\omega \in \Omega$ and $x_1,x_2\in X$ with $x_1 \leq x_2$, then
$u(t,\omega,x_1) \leq u(t,\omega,x_2)$ for each $t$ in the common interval of
definition of the trajectories. We denote $\omega{\cdot}t =\sigma(t,\omega)$ and
$u$ satisfies the cocycle identity
$u(t+s,\omega,x)=u(t,\omega{\cdot}s,u(s, \omega , x))$ for every $t,s \geq 0$.
\par
It is well known that the skew-product formalism is a powerful tool in the study
of linear and nonlinear evolution systems.  Frequently, this formalism is
obtained from a single nonautonomous differential equation using a standard hull
construction.
\par
An important result for monotone uniformly stable recurrent skew-product
semiflows is the convergence of relatively compact trajectories to their
omega-limit sets, which define 1-coverings of the base space. This result was
firstly proved by Jiang and Zhao~\cite{paper:jizh} in an abstract setting
applicable to cooperative systems of ordinary, finite delay and parabolic
nonautonomous differential equations. In particular, when the skew-product
semiflow comes from a single differential equation with a recurrence property in
the coefficients as constancy, periodicity, almost-periodicity among others, it
provides a unified generalization of the asymptotic constancy, periodicity or
almost-periodicity of the solutions, studied in many previous papers. \par
This theory was extended to nonautonomous functional differential equations
(FDEs for short) with infinite delay in Novo {\it et al.}~\cite{paper:NOS2007}
(see also Wang and Zhao~\cite{paper:wazh} for other implications of this
theory). In that paper, $X=BU$, i.e. the subset of functions of
$C((-\infty,0],\R^m)$ that are bounded and uniformly continuous, and the
semiflow is generated by the solutions of a family $y'=F(\w{\cdot}t,y_t)$ of
FDEs defined by a continuous function $F:\Omega \times BU \rightarrow BU$ which
is locally Lipschitz in its second variable. The space $BU$ satisfies standard
conditions of regularity that imply existence, uniqueness and continuous
dependence of the solutions with respect to the initial data (see Hino {\it et
  al.}~\cite{book:hino}).\par
Later, motivated by the applicability to the study of the long-term dynamics of
compartmental systems, the paper by Mu\~{n}oz-Villarragut {\it et
  al.}~\cite{paper:MNO} is the starting point of an important effort to extend
the previous results to nonautonomous neutral functional differential equations
(NFDEs for short) with infinite delay.
\par
Compartmental systems have been used as mathematical models for the study of the dynamical behavior of many processes in the biological and physical sciences
which depend on local mass balance conditions (see Jacquez~\cite{paper:jacq},
Jacquez and Simon~\cite{paper:jasi} and the references therein). Some initial
results for models described by FDEs with finite and infinite delay can be found
in Gy\H{o}ri~\cite{paper:gyor}, and Gy\H{o}ri and Eller~\cite{paper:gyel}. The
papers by Arino and Bourad~\cite{paper:arbo}, and Arino and
Haourigui~\cite{paper:arha} prove the existence of almost periodic solutions for
compartmental systems described by almost periodic FDEs and NFDEs with finite
delay. Gy\H{o}ri and Wu~\cite{paper:gywu} modeled the dynamical properties of
compartmental systems with active compartments by means of NFDEs with infinite
delay, whose neutral term represents the net amount of material produced or
swallowed by the compartments. This type of NFDEs equations have been
investigated by Wu and Freedman~\cite{paper:wufr}, and Wu~\cite{paper:jwu1991}.
\par
An important difficulty that appears in the monotone theory of NFDEs is that, in
many applications, the order structure must be defined by means of an
exponential ordering which provides a positive cone with empty
interior. Krisztin and Wu~\cite{paper:krwu} show the asymptotic periodicity of the solutions with Lipschitz continuous initial
data,  under appropriate conditions on the
coefficients of scalar periodic NFDEs with finite delay and linear neutral term,
which imply the monotonicity of the solutions for an exponential ordering.  By means of monotone skew-product semiflow techniques, Novo {\it et
  al.}~\cite{paper:NOV}, and Obaya and Villarragut~\cite{paper:obvi2} generalize
the previous results obtaining that, under appropriate assumptions, a family of
nonautonomous NFDEs with infinite delay and linear neutral term induces a
monotone skew-product semiflow on $\Omega \times BU$ for the exponential
ordering, and the omega-limit sets of bounded trajectories with Lipschitz
continuous initial data are copies of the base.  The case of stable
nonautonomous operator $D$ for the neutral part is also considered in Obaya and
Villarragut~\cite{paper:obvi}, where similar results are obtained for a new
transformed exponential ordering.
\par
The present paper provides new contributions to the core of the dynamical theory
of monotone recurrent skew-product semiflows generated by FDEs and NFDEs with
infinite delay, and improves the conditions of applicability of the theory to
compartmental systems and other models of interest, which will be explained in detail in forthcoming publications.
More precisely, this work provides a dynamical framework to study compartmental
systems described by neutral functional differential equations analogous to those
considered in~\cite{paper:gywu, paper:krwu,paper:NOS2007, paper:NOV, paper:obvi, paper:obvi2},
under physical conditions that have not been previously considered in the literature. In addition, Wu and and Zhao~\cite{paper:WZ2002} introduce the exponencial ordering and the associated monotone methods for abstract delayed reaction diffusion equations, and show a natural
way to extend the conclusions of the above references to nonautonomous compartmental systems with spatial diffusion, to which our study may also be applied.
\par%%%%%%%%%%%%%%%%%%%%%%%%%%%%%%%%%%%%%%%%%%%%%%%%%%%%%%%%%%%%%%%%%%%%%%%
The structure and main goals of the paper are now described. Some basic notions
and properties of the theory of nonautonomous dynamical systems are included in
Section~\ref{secprel}. Section~\ref{secfun} is devoted to the study of families
$y'=F(\w{\cdot}t,y_t)$ of FDEs with infinite delay defined by continuous
functions $F:\Omega \times BC \rightarrow BC$, where
$BC =\{y\in C((-\infty,0],\R^m) \mid y\text{ is bounded}\}$, which are locally
Lipschitz continuous in their second variable. Although initial data in
$\Omega \times BC$ are physically admissible, the choice of this set as a phase
space is problematic, because the existence of a solution of the Cauchy problem
requires the measurability of the map
$(-\infty,T] \rightarrow \R^m, t \rightarrow F(\omega \cdot t,y_t)$ for each
$T \in \R$, $ \omega \in \Omega $ and $y \in C((-\infty,T],\R^m]$, (see
Driver~\cite{paper:driver} and Seifert~\cite{paper:seifert}). In our setting,
this is a consequence of assuming the continuity of
$F:\Omega \times B_r \rightarrow \R^m$ when the closed ball $B_r \subset BC$ is
endowed with the compact open topology, which, of course, is satisfied in all the  physical models that we want to work with. Instead of considering Lipschitz
continuous initial data in $BC$, we introduce the bigger set ${\mathcal{R}}$ of
the elements in $BC$ with uniformly bounded variation on the intervals
$[-k,-k+1]$ for $k \geq 1$. By considering an appropriate exponential ordering
$\leq_A$, assuming a quasimonotone condition on $F$, a componentwise separation
property and the uniform stability of $B_r$ for the order $\leq_A$, the main
conclusion of this section is that omega-limit sets of bounded trajectories with
initial data in ${\mathcal{R}}$ are $1$-coverings of the base $\Omega$, that is, the recurrent character is inherited.\par
Section~\ref{neutral} extends, for the transformed exponential ordering
introduced in~\cite{paper:obvi}, the previous results to families
$\frac{d}{dt}D(\omega{\cdot}t,z_t)=G( \omega{\cdot} t,z_t)$ of NFDEs defined by
a stable neutral term $D:\Omega \times BC \rightarrow \R^m$, linear in the state
component, and a function $G:\Omega \times BC \rightarrow \R^m$ that satisfies
properties of regularity analogous to those considered in the previous section.
The main idea is to deduce, from the stability of $D$, the invertibility of the
operator ${\widehat D}\colon\Omega \times BC\to \Omega \times BC$,
$(\w,x)\mapsto (\w, \wh D_2(\w,x))$, where
$\wh D_2(\w,x)\colon(-\infty,0]\to\R^m$, $s\mapsto D(\w{\cdot}s,x_s)$, and to
transform the NFDE into a FDE to which the conclusions of Section~\ref{secfun}
can be applied. As a consequence, the omega-limit sets of bounded trajectories with initial datum $x$ satisfying $\wh D_2(\w,x)\in {\mathcal{R}}$ are $1$-coverings of the
base $\Omega$, or what is equivalent, the trajectories reproduce asymptotically the recurrent behavior of the coefficients of the neutral and non-neutral part, that is, the dynamics exhibited by the time variation of the equation.
\par %%%%%%%%%%%%%%%%%%%%%%%%%%%%%%%%%%%%%%%%%%%%%%%%%%%%%%%%%%%%%%%%%%
%%%%%%%%%%%%%%%%%%%%%%%%%%%%%%%%%%%%%%%%%%%%%%%%%%%%%%%%%%%%%%%%%%%%%%
%%%%%%%%%%%%%%%%%%%%%%%%%%%%%%%%%%%%%%%%%%%%%%%%%%%%%%%%%%%%%%%%%%%%%%
\section{Some preliminaries}\label{secprel}
Let $(\Om,d)$ be a compact metric space. A real {\em continuous flow \/}
$(\Om,\sigma,\R)$ is defined by a continuous mapping
$\sigma: \R\times \Om \to \Om,\; (t,\w)\mapsto \sigma(t,\w)$ satisfying
\begin{enumerate}
  \renewcommand{\labelenumi}{(\roman{enumi})}
\item $\sigma_0=\text{Id},$
\item $\sigma_{t+s}=\sigma_t\circ\sigma_s$ for each $s$, $t\in\R$,
\end{enumerate}
where $\sigma_t(\w)=\sigma(t,\w)$ for all $\w \in \Om$ and $t\in \R$. The set
$\{ \sigma_t(\w) \mid t\in\R\}$ is called the {\em orbit\/} or the {\em
  trajectory\/} of the point $\w$. We say that a subset $\Om_1\subset \Om$ is
{\em $\sigma$-invariant\/} if $\sigma_t(\Om_1)=\Om_1$ for every $t\in\R$.  A
subset $\Om_1\subset \Om$ is called {\em minimal \/} if it is compact,
$\sigma$-invariant and its only nonempty compact $\sigma$-invariant subset is
itself. Every compact and $\sigma$-invariant set contains a minimal subset; in
particular, it is easy to prove that a compact $\sigma$-invariant subset is
minimal if and only if every trajectory is dense. We say that the continuous
flow $(\Om,\sigma,\R)$ is {\em recurrent\/} or {\em minimal\/} if $\Om$ is
minimal. Almost periodic and almost automorphic flows are relevant examples of
recurrent flows. We refer to Ellis~\cite{book:Ellis} and see Shen and
Yi~\cite[part II]{book:shyi} for the study of topological and ergodic properties
of these flows.
\par
Let $E$ be a complete metric space and $\R^+=\{t\in\R\,|\,t\geq 0\}$. A {\em
  semiflow} $(E,\Phi,\R^+)$ is determined by a continuous map
$\Phi: \R^+\times E \to E,\; (t,x)\mapsto \Phi(t,x)$ which satisfies
\begin{enumerate}
  \renewcommand{\labelenumi}{(\roman{enumi})}
\item $\Phi_0=\text{Id},$
\item $\Phi_{t+s}=\Phi_t \circ \Phi_s\;$ for all $\; t$, $s\in\R^+,$
\end{enumerate}
where $\Phi_t(x)=\Phi(t,x)$ for each $x \in E$ and $t\in \R^+$.  The set
$\{ \Phi_t(x)\mid t\geq 0\}$ is the {\em semiorbit\/} of the point $x$. A subset
$E_1$ of $E$ is {\em positively invariant\/} (or just $\Phi$-{\em invariant\/})
if $\Phi_t(E_1)\subset E_1$ for all $t\geq 0$. A semiflow $(E,\Phi,\R^+)$ admits
a {\em flow extension\/} if there exists a continuous flow $(E,\wt \Phi,\R)$
such that $\wt \Phi(t,x)=\Phi(t,x)$ for all $x\in E$ and $t\in\R^+$. A compact
and positively invariant subset admits a flow extension if the semiflow
restricted to it admits one.
\par
Write $\R^-=\{t\in\R\,|\,t\leq 0\}$. A {\em backward orbit\/} of a point
$x\in E$ in the semiflow $(E,\Phi,\R^+)$ is a continuous map $\psi:\R^-\to E$
such that $\psi(0)=x$ and, for each $s\leq 0$, it holds that
$\Phi(t,\psi(s))=\psi(s+t)$ whenever $0\leq t\leq -s$. If for $x\in E$ the
semiorbit $\{\Phi(t,x)\mid t\ge 0\}$ is relatively compact, we can consider the
{\em omega-limit set\/} of $x$,
\[
  \mathcal{O}(x)=\bigcap_{s\ge 0}{\rm closure}{\{\Phi(t+s,x)\mid t\ge 0\}}\,,
\]
which is a nonempty compact connected and $\Phi$-invariant set.  Namely, it
consists of the points $y\in E$ such that $y=\lim_{n\to \infty} \Phi(t_n,x)$ for
some sequence $t_n\uparrow \infty$. It is well-known that every
$y\in\mathcal{O}(x)$ admits a backward orbit inside this set. Actually, a
compact positively invariant set $M$ admits a flow extension if every point in
$M$ admits a unique backward orbit which remains inside the set $M$
(see~\cite[part II]{book:shyi}).
\par
A compact positively invariant set $M$ for the semiflow $(E,\Phi,\R^+)$ is {\em
  minimal\/} if it does not contain any other nonempty compact positively
invariant set than itself. If $E$ is minimal, we say that the semiflow is
minimal.
\par
A semiflow is {\em of skew-product type\/} when it is defined on a vector bundle
and has a triangular structure; more precisely, a semiflow
$(\Om\times X,\tau,\,\R^+)$ is a {\em skew-product\/} semiflow over the product
space $\Om\times X$, for a compact metric space $(\Om,d)$ and a complete metric
space $(X,\textsf{d})$, if the continuous map $\tau$ is as follows:
\begin{equation}\label{skewp}
  \begin{array}{cccl}
    \tau \colon  &\R^+\times\Om\times X& \longrightarrow & \Om\times X \\
                 & (t,\w,x) & \mapsto &(\w{\cdot}t,u(t,\w,x))\,,
  \end{array}
\end{equation}
where $(\Om,\sigma,\R)$ is a real continuous flow
$\sigma:\R\times\Om\rightarrow\Om$, $\,(t,\w)\mapsto \w{\cdot}t$, called the
{\em base flow\/}. The skew-product semiflow~\eqref{skewp} is {\em linear\/} if
$u(t,\w,x)$ is linear in $x$ for each $(t,\w)\in\R^+\times\Om$.
\section{Functional differential equations with infinite delay}\label{secfun}
We consider the Fr\'{e}chet space $X=C((-\infty,0],\R^m)$ endowed with the
compact-open topology, i.e.~the topology of uniform convergence over compact
subsets, which is a metric space for the distance
\[
  \di(x,y)=\sum_{n=1}^\infty \frac{1}{2^n}\frac {\n{x-y}_n}{1+\n{x-y}_n}\,,\quad
  x,y\in X\,,
\]
where $\n{x}_n=\sup_{s\in[-n,0]}\n{{x(s)}}$, and $\n{\cdot}$ denotes the maximum
norm in $\R^m$. \par
Let $(\Om,\sigma,\R)$ be a minimal flow over a compact metric space $(\Om,d)$
and denote $\sigma(t,\w)=\w{\cdot}t$ for each $\w \in\Omega$ and $t\in\R$.  As
usual, given $I=(-\infty,a]\subset\R$, $t\in I$ and a continuous function
$x:I\to\R^m$, $x_t$ will denote the element of $X$ defined by $x_t(s)=x(t+s)$
for $s\in (-\infty,0]$.  We consider the family of nonautonomous infinite delay
functional differential equations
\begin{equation}\label{infdelay}
  z'(t)=F(\w{\cdot}t,z_t)\,, \quad t\geq 0\,,\;\w\in\Omega\,.
\end{equation}
The first objective of this section is to provide an appropriate framework to
study the dynamical behavior of the solutions of~\eqref{infdelay}. One of the
admissible phase spaces for the study of these equations is $BU$ (see~\cite{book:hino}), the Banach
space of bounded and uniformly continuous functions in $X$, i.e.
\[
  BU =\{x\in X \mid x \text{ is bounded and uniformly continuous}\}
\]
with the supremum norm $\n{x}_\infty=\sup_{s\in(-\infty,0]} \n{x(s)}$.\par
This is not the case for the Banach space
\[
  BC =\{x\in X \mid x \text{ is bounded}\}
\]
where, in general, the family~\eqref{infdelay} does not induce a local
skew-product semiflow on $\R^+\times \Om\times BC$. However, in many
applications, initial data in $BC$ are physically admissible. We show how to
overcome this drawback to introduce a dynamical structure on $\Omega\times BC$.
\par
Given $r>0$, we will denote
\[
  B_r=\{x\in BC \mid \n{x}_\infty \leq r\}
\]
and we consider the family of nonautonomous FDEs~\eqref{infdelay} defined by a
function $F\colon\Omega\times BC \to\R^m$, $(\w,x)\mapsto F(\w,x)$ satisfying:
\begin{enumerate}[label=\upshape(\textbf{F\arabic*}),series=infdelay_properties,leftmargin=27pt]
\item\label{F1} $F$ is continuous on $\Om\times BC$  when the the norm $\n{\cdot}_\infty$ is considered on $BC$, and Lipschitz continuous on  $\Om\times B_r$ in its second variable for   each $r>0$,
\end{enumerate}
which in particular implies that
\begin{equation}\label{F2}
  \text{$F(\Om\times B_r)$   is a bounded subset of $\R^m$ for each $r>0$.}
\end{equation}
From this condition, the standard theory of infinite delay differential
equations (see~\cite{book:hino}) assures that, for each $x\in BU$ and each
$\w\in\Om$, the system~\eqref{infdelay}$_\w$ locally admits a unique solution
$z(\cdot,\w,x)$ with initial value $x$, i.e.~$z(s,\w,x)=x(s)$ for each
$s\in (-\infty,0]$. Therefore, the family~\eqref{infdelay} induces a local
skew-product semiflow
\begin{equation}\label{skewBU}
  \begin{array}{cccl}
    \tau &:\mathcal{U}\subset \R^+\times\Om\times BU& \longrightarrow & \Om\times BU\\
         & (t,\w,x) & \mapsto &(\w{\cdot}t,u(t,\w,x))\,,
  \end{array}
\end{equation}
where $u(t,\w,x)\in BU$ and $u(t,\w,x)(s)=z_t(\omega,x)(s)=z(t+s,\w,x)$ for
$s\in (-\infty,0]$.\par
When the initial data $x$ belongs to $BC$ the existence and uniqueness is not
guaranteed from~\ref{F1}. In addition, we impose the following condition satisfied in important applications, such as compartmental systems and neural networks.
\begin{enumerate}[resume*=infdelay_properties]\setlength\itemsep{0.4em}
\item\label{F3} for each $r>0$, $F\colon\Om\times B_r\to\R^m$ is continuous when
  we take the restriction of the compact-open topology to $B_r$, i.e.~if
  $\w_n\to\w$ and $x_n\stackrel{\di\;}\to x$ as $n\uparrow\infty$ with
  $x\in B_r$, then $\lim_{n\to\infty}F(\w_n,x_n)=F(\w,x)$.
\end{enumerate}
Next, we show that $\Omega\times BC$ is indeed a good phase space.
\begin{proposition}\label{pro:exitBC}
  Under assumptions~\ref{F1}--\ref{F3}, for each $x\in BC$ and each $\w\in\Om$
  the system~\eqref{infdelay}$_\w$ locally admits a unique solution
  $z(\cdot,\w,x)$ with initial value $x$, i.e.~$z(s,\w,x)=x(s)$ for each
  $s\in (-\infty,0]$.
\end{proposition}
\begin{proof}
  As explained in~\cite{paper:seifert} (see also Sawano~\cite{paper:sawano}),
  the result can be deduced from~\cite{paper:driver} once we check the
  continuity of the map
  \[(-\infty,T]\to \R^m,\quad t\mapsto F(\w{\cdot}t,y_t)\] for each $T \in \R$
  and each bounded function $y\in C((-\infty,T],\R^m)$. This is an easy
  consequence of~\ref{F3}, because if $t=\lim_{n\to\infty}t_n$ with
  $t_n\in(-\infty,T]$, then $y_{t_n}$ belongs to some $B_r$ for all $n\in\N$ and
  $y_{t_n}\stackrel{\di\;}\to y_{t}$ as $n\uparrow \infty$.
\end{proof}As a consequence, the map~\eqref{skewBU} is extended to
$\Om\times BC$. Moreover, as shown next, this extension turns out to be
continuous on bounded sets when the restriction of the compact-open topology to
$BC$ and the product metric topology on $\Omega\times BC$ are considered.
\begin{proposition}\label{continuitybolas} Under
  assumptions~\ref{F1}--\ref{F3}, the local map
  \begin{equation*}
    \begin{array}{ccl}
      \mathcal{U}\subset\R^+\times\Om\times B_r& \longrightarrow & \Om\times BC\\
      (t,\w,x) & \mapsto &(\w{\cdot}t,u(t,\w,x))
    \end{array}
  \end{equation*}
  is continuous when we take the restriction of the compact-open topology to
  $B_r$, i.e.~if $t_n\to t$, $\w_n\to\wt\w$ and
  $x_n\stackrel{\textup\di\;}\to \wt x$ as $n\uparrow\infty$ with
  $x_n,\, \wt x\in B_r$ for all $n\in\N$, then $\w_n{\cdot}t_n\to \wt\w{\cdot}t$
  and $u(t_n,\w_n,x_n)\stackrel{\textup\di\;}\to u(t,\wt \w,\wt x)$ as
  $n\uparrow\infty$.
\end{proposition}
\begin{proof} First we fix a $t\in\R^+$ such that $u(t,\wt \w,\wt x)$ is defined
  and we check that $u(t,\w_n,x_n)\stackrel{\di\;}\to u(t,\wt\w,\wt x)$ as
  $n\uparrow \infty$. If $F$ is a bounded function on $\Om\times BC$, then
  $\sup_{\tau\in[0,t], n\geq 1}\n{u(\tau,\w_n,\,x_n)}_\infty<\infty$ and the
  proof of Proposition~4.2 of~\cite{paper:NOS2007} can be easily adapted to this
  case. Otherwise, take $\delta>0$ such that
  $u(\tau,\wt \w,\wt x)$ is defined for $\tau \in[0,t+\delta]$ and denote by
  $k=\sup_{\tau\in[0,t+\delta]}\{\n{u(\tau,\wt \w,\wt x)}_\infty,r\}$.  In
  addition, from~\eqref{F2}, we know that $F(\Om\times B_{k+1})$ is bounded and
  we can take $\rho=\sup_{(\w,x)\in \Om\times B_{k+1}}\n{F(\w,x)}$. Now let
  $\varphi\colon \R^m\to\R^m$ be a $C^\infty$ function such that
  \begin{equation*}
    \varphi(y)=\begin{cases} y, & \text{if } \n{y}\leq \rho\,,\\
      0,  &\text{if } \n{y} \geq \rho +1\ \end{cases}
  \end{equation*}
  and consider the family of equations
  \begin{equation*}
    \wt z{\,'}(t)=\varphi(F(\w{\cdot}t,\wt z_t))\,, \quad t\geq 0\,,\;\w\in\Omega\,.
  \end{equation*}
  The boundedness of $\varphi\circ F$ provides
  $v(\tau,\w_n,x_n)\stackrel{\di\;}\to v(\tau ,\wt\w,\wt x)$ as
  $n\uparrow \infty$ for each $\tau\in [0,t+\delta]$, where
  $v(\tau,\w,x)( s)=\wt z(\tau+ s,\w,x)$ for $ s\in (-\infty,0]$, as
  usual. Therefore, the definitions of $\varphi$ and $\rho$ yield
  $v(\tau,\wt \w,\wt x)=u(\tau,\wt \w,\wt x)$ for $\tau\in[0,t+\delta]$, that
  is,
  \begin{equation*}
    \wt z(\tau ,\w_n,x_n) \to z(\tau,\wt \w,\wt x) \text{ as } n\uparrow \infty \text{ uniformly for }\tau \in  [0,t+\delta]\,.
  \end{equation*}
  From this, together with $\n{x_n}_\infty\leq r$ for each $n\in\N$, we deduce
  that there is an $n_0\in\N$ such that
  $\sup_{\tau\in[0,t+\delta]} \n{v(\tau,\w_n,x_n)}_\infty\leq k+1$ for each
  $n\geq n_0$. Hence, $v(t,\w_n,x_n)=u(t,\w_n,x_n)$ for $n\geq n_0$ and
  $u(t,\w_n,x_n)\stackrel{\di\;}\to u(t,\wt\w,\wt x)$ as $n\uparrow \infty$, as
  claimed. Moreover, $v(\tau,\w_n,x_n)=u(\tau,\w_n,x_n)$ for $n\geq n_0$ and
  $\tau\in [0,t+\delta]$, and
  \begin{equation}\label{cotaene}
    \sup\{\n{u(\tau,\w_n,x_n)}_\infty \mid \tau\in [0,t+\delta]\,,\; n\in\N\}<\infty\,.
  \end{equation}
  \par
  Finally, if in addition $t_n\to t$ as $n\uparrow \infty$, then
  \begin{equation*}
    \di(u(t,\wt\w,\wt x), u(t_n,\w_n,x_n)) \leq \di(u(t,\wt\w,\wt x), u(t,\w_n,x_n))+ \di(u(t,\w_n,x_n), u(t_n,\w_n,x_n))
  \end{equation*}
  and we only have to check that the second term vanishes as $n\uparrow
  \infty$. \par
  Let $[a,b]\subset (-\infty, 0]$ and $s\in [a,b]$. If $t+s> 0$, we take an
  $n_1\in\N$ such that the real interval $I_n$ with extrema $t+s$ and $t_n+s$ is
  contained in $(0,t+\delta)$ for $n\geq n_1$.  Thus, from~\eqref{cotaene}
  and~\eqref{F2}, we deduce that there is a constant $M$ such that
  \begin{equation}\label{t+s>0}
    \!\!\n{u(t,\w_n,x_n)(s)-u(t_n,\w_n,x_n)(s)}\leq \int_{I_n}\!\! \n{F(\w_n{\cdot}\tau, z_\tau(\w_n, x_n))}\, d\tau\leq M\, |t-t_n|
  \end{equation}
  for each $n\geq n_1$.  If $t+s < 0$, there is also an $n_2\geq n_1$ such that
  $t_n+ s< 0$ for each $n\geq n_2$ and, thus,
  \begin{equation}\label{t+s<0}
    \n{u(t,\w_n,x_n)(s)-u(t_n,\w_n,x_n)(s)}=\n{x_n(t+s)-x_n(t_n+s)}\,.
  \end{equation}
  We omit the case $t+s=0$ because it is a combination of the previous
  cases. Therefore, from~\eqref{t+s>0}, \eqref{t+s<0} and the convergence of
  $x_n\stackrel{\di\;}\to \wt x$ as $n\uparrow\infty$, it is easy to check that
  $\n{u(t,\w_n,x_n)(s)-u(t_n,\w_n,x_n)(s)}$ converges to $0$ as
  $n\uparrow \infty$ uniformly for $s$ in the compact set $[a,b]$, which
  finishes the proof.
\end{proof}
The next result proves that, under assumptions~\ref{F1} and~\ref{F3}, each
bounded solution $z(\cdot,\w_0,x_0)$ provides a relatively compact
trajectory. Note that solutions that remain bounded are globally defined on the
whole real line (see e.g.~\cite{paper:sawano}).
\begin{proposition}\label{trajectory}
  Assume~\ref{F1}--\ref{F3}. If $x_0\in BC$ and $z(\cdot ,\w_0,x_0)$ is a
  solution of equation~\eqref{infdelay}$_{\w_0}$ bounded for the norm
  $\n{\cdot }_\infty $, then $\mathcal{F}=\{u(t,\w_0,x_0)\mid t\ge 0\}$ is a
  relatively compact subset of $BC$ for the compact-open topology.
\end{proposition}
\begin{proof}
  Let $r=\sup_{t\geq 0}\n{u(t,\w_0,x_0)}_\infty$.  According to Theorem~8.1.4 of~\cite{book:hino}, $\mathcal{F}$ is relatively compact in $X$
  if, and only if, for every $s\in(-\infty,0]$ $\mathcal{F}$ is equicontinuous
  at $s$ and $\mathcal{F}(s)=\{u(t,\w_0,x_0)(s)\mid t\geq 0\}$ is relatively
  compact in $\R^m$.\par
  The second condition holds because $\mathcal{F}\subset B_r$.  As for the
  equicontinuity, let $\rho>0$, $\eps>0$ and
  $M=\sup_{(\w,x)\in \Om\times B_r}\n{F(\w,x)}$, which is finite thanks
  to~\eqref{F2}.  Then, for each $t\ge \rho$ and $s_1,s_2\in[-\rho ,0]$ with
  $|s_1-s_2|<\eps/M$ and $s_1\leq s_2$ (the case $s_2\leq s_1$ is analogous), we
  have
  \begin{equation}\label{functional:relcomptraj}
    \n{u(t,\w_0,x_0)(s_1)-u(t,\w_0,x_0)(s_2)}\le
    \int_{t+s_1}^{t+s_2} \n{F(\w_0{\cdot}\tau,z_\tau (\w,_0,x_0))}\,d
    \tau\le \eps\,.
  \end{equation}
  On the other hand, if $t \in [0,\rho ]$, then, for each $s\in[-\rho ,0]$,
  \[
    t+s-\rho \in[-2\rho ,0]\quad\text{and}\quad u(t,\w_0,x_0)(s)=
    u(\rho,\w_0,x_0)(t+s-\rho)\,.
  \]
  Therefore, the equicontinuity of $\mathcal{F}$ follows
  from~\eqref{functional:relcomptraj} and the uniform continuity of
  $u(\rho,\w_0,x_0)$ on $[-2\rho ,0]$, which finishes the proof.
\end{proof}
In the situation of the foregoing proposition, we can define the
\emph{omega-limit set} of the trajectory of the point $(\w_0,x_0)$ as
\[\mathcal{O}(\w_0,x_0)=\big\{(\w,x)\mid \exists \,t_n\uparrow
  \infty\;\textup{ with } \,\w_0{\cdot}t_n\to\w\,,\;
  u(t_n,\w_0,x_0)\stackrel{\di\;}\to x\big\}\,,\] and the following proposition
provides its main properties.
\begin{proposition}\label{omegalimit} Assume~\ref{F1}--\ref{F3}.
  If $(\w_0,x_0)\in\Om\times BC$ and $z({\cdot},\w_0,x_0)$ is a solution
  of~\eqref{infdelay}$_{\w_0}$ bounded for the norm $\n{{\cdot}}_\infty$, then
  $\mathcal{O}(\w_0,x_0)$ is a nonempty, compact and invariant subset of
  $\Om\times BU$ admitting a flow extension.
\end{proposition}
\begin{proof} Thanks to Proposition~\ref{trajectory}, $\mathcal{O}(\w_0,x_0)$ is
  nonempty and relatively compact; in order to prove that $\mathcal{O}(\w_0,x_0)$ is compact, it
  suffices to check that it is closed, which is omitted.
  \par
  Next we show that $\mathcal{O}(\w_0,x_0)\subset \Om\times BU$.  Let
  $r=\sup_{t\geq 0}\n{u(t,\w_0,x_0)}_\infty$ and
  $M=\sup_{\tau\geq 0}\n{F(\w_0{\cdot}\tau,z_\tau(\w_0,x_0)}$, which is finite
  from~\eqref{F2} because $z_\tau(\w_0,x_0)=u(\tau,\w_0,x_0)\in B_r$.  Take
  $(\w,x)\in \mathcal{O}(\w_0,x_0)$, i.e.
  \begin{equation}\label{omegalimitcon}
    \exists \,t_n\uparrow
    \infty\;\textup{ with } \w=\lim_{n\to\infty}\w_0{\cdot}t_n \textup{ and }
    x\stackrel{\di\;}=\lim_{n\to\infty} u(t_n,\w_0,x_0)\,.
  \end{equation}
  Then, given $t,s\in(-\infty,0]$ (assume without loss of generality that
  $t\leq s$), there is an $n_0\in\N$ depending on them such that $t_n+t\geq 0$
  and $t_n+s\geq 0$ for each $n\geq n_0$.  Then, we have
  \[\n{z(t+t_n,\w_0,x_0)-z(t_n+s,\w_0,x_0)}\leq\int_{t_n+t}^{t_n+s}\!\!\!
    \n{F(\w_0{\cdot}\tau, z_\tau(\w_0, x_0))}\, d\tau\leq M\, |t-s|\,,\] which
  in turn implies that
  \begin{equation*}
    \n{x(t)-x(s)}\leq \lim_{n\to\infty}\n{z(t+t_n,\w_0,x_0)-z(t_n+s,\w_0,x_0)}\leq M \, |t-s|\,
  \end{equation*}
  and proves that $x\in BU$, as claimed.  The positive invariance, i.e.
  $\tau_t(\mathcal{O}(\w_0,x_0))\subset \mathcal{O}(\w_0,x_0)$ for each $t>0$,
  is deduced from Proposition~\ref{continuitybolas} as follows:
  \[
    \begin{array}{l}
      \displaystyle \w=\lim_{n\to\infty}\w_0{\cdot}t_n\\
      \displaystyle x\stackrel{\di\;}=\lim_{n\to\infty} u(t_n,\w_0,x_0)
    \end{array} \Longrightarrow
    \begin{array}{l}
      \displaystyle \w{\cdot}t=\lim_{n\to\infty}\w_0{\cdot}(t+t_n)\\
      \displaystyle u(t,\w,x)\stackrel{\di\;}=\lim_{n\to\infty} u(t+t_n,\w_0,x_0)
    \end{array}
  \]
  because $u(t_n,\w_0,x_0)\in B_r$ and
  $u(t,\w_0{\cdot}t_n,u(t_n,\w_0,x_0))=u(t+t_n,\w_0,x_0)$, $n\in\N$.
  \par
  Let us check that, in fact,
  $\tau_t(\mathcal{O}(\w_0,x_0))= \mathcal{O}(\w_0,x_0)$ for each $t>0$,
  i.e. $\mathcal{O}(\w_0,x_0)$ is invariant. Fix $t>0$ and
  $(\w,x)\in \mathcal{O}(\w_0,x_0)$,
  i.e. satisfying~\eqref{omegalimitcon}. Since there is an $n_0$ such that
  $t_n-t\geq 0$ for each $n\geq n_0$, from Proposition~\ref{trajectory} we
  deduce that there exists a subsequence, which will be also denoted by
  $\{t_n\}_n$, and $(\w_1,x_1)\in \mathcal{O}(\w_0,x_0)$ such that
  \[
    \w_1=\lim_{n\to\infty}\w_0{\cdot}(t_n-t) \quad\text{ and } \quad
    x_1\stackrel{\di\;}= \lim_{n\to\infty} u(t_n-t,\w_0,x_0)\,.
  \]
  Finally, as above from Proposition~\ref{continuitybolas} we get
  \[
    \qquad \qquad\!  \w_1{\cdot}t=\lim_{n\to\infty}\w_0{\cdot}t_n=\w\quad
    \text{and}\quad u(t,\w_1,x_1) \stackrel{\di\;}= \lim_{n\to\infty}
    u(t_n,\w_0,x_0)=x \,,
  \]
  and $(\w,x)\in\tau_t(\mathcal{O}(\w_0,x_0))$, as desired.
  \par
  Once we have proved that $\mathcal{O}(\w_0,x_0)\subset \Om\times B_r$ is
  invariant, again from Proposition~\ref{continuitybolas} we deduce that the
  semiflow $\tau$ is continuous on $\R^+\times \mathcal{O}(\w_0,x_0)$ when the
  product metric topology on $\mathcal{O}(\w_0,x_0)$ is taken. To see that the
  semiflow over $\mathcal{O}(\w_0,x_0)$ admits a flow extension, from
  Theorem~2.3 (part II) of~\cite{book:shyi} it suffices to show that
  every point in $\mathcal{O}(\w_0,x_0)$ admits a unique backward orbit which
  remains inside the set $\mathcal{O}(\w_0,x_0)$. See Proposition~4.4
  of~\cite{paper:NOS2007} for the details.
\end{proof}
As explained before, this paper provides a contribution to the dynamical theory
of monotone recurrent skew-product semiflows. We consider a monotone structure
on $\Omega\times BC$ determined by an exponential ordering and we enhance the
theory started in~\cite{paper:jizh}, \cite{paper:NOS2007} and \cite{paper:MNO},
where the 1-covering property of omega-limit sets of relatively compact
trajectories was proved.
\par%%%%%%%%%%%%%%%%%%%%%%%%%%%%%%%%%%%%%%%%%%%%%%%%%%%%%%%%%%%%%%%%%%%%%%%%%%%%%%
Let $A$ be a diagonal matrix with negative diagonal entries
$a_1,\ldots,a_m$. Notice that such $A$ is a \emph{quasipositive} matrix,
i.e. there exists $\lambda>0$ such that $A+\lambda I$ is a matrix  whose
entries are  all nonnegative. As in~\cite{paper:NOV}, considering
the componentwise partial ordering on $\mathbb{R}^m$, we introduce the positive
cone with empty interior in $BC$
\begin{align*}
  BC^+_A&=\{ x\in BC\mid x\geq 0 \;\text{ and }\; x(t)\geq
          e^{A(t-s)}x(s)\quad
          \text{for } -\infty<s\leq t\leq 0\}\\
        &=\{ x\in BC\mid x\geq 0 \;\text{ and }\; t\mapsto
          e^{-A\,t}x(t)\text{ is a nondecreasing function}\}\,,
\end{align*}
which induces the following partial order relation on $BC$:
\begin{align}\label{orderA}
 \hspace{-.4cm} x\le_A y\;&\Longleftrightarrow\; x\le y \;\text{ and }\;
              y(t)-x(t)\geq e^{A(t-s)}(y(s)-x(s))\,,
              -\infty<s\leq t\leq 0\,,\nonumber \\
  \hspace{-.4cm}  x<_A y \;&\Longleftrightarrow\; x \leq_A y\;\;\text{and}\quad
             x\neq y\,.
\end{align}
Let us assume one additional quasimonotone condition on $F$:
\begin{enumerate}[resume*=infdelay_properties]
\item\label{F4} If $x,y\in BC$ with $x\leq_A y$, then
  $F(\w,y)-F(\w,x)\geq A\,(y(0)-x(0))$ for each $\w\in\Om$ and the above quasipositive matrix $A$.
\end{enumerate}
\par
From this hypothesis, the monotone character of the semiflow~\eqref{skewBU} and
its extension to $\Om\times BC$ are deduced. We omit the proof, analogous to
that of Proposition~3.1 of Smith and Thieme~\cite{paper:smth2}.
\begin{theorem}\label{preliminaries:monotone}
  Under assumptions~\ref{F1}--\ref{F4}, for each $\w\in\Omega$ and $x$,
  $y\in BC$ such that $x\leq_A y$, it holds that
  \[u(t,\w,x)\leq_A u(t,\w,y)\] for all $t\ge 0$ where they are defined.
\end{theorem}
Next, let us recall the definition of uniform stability for the order $\leq_A$.
\begin{definition}
  A subset $K$ of $BC$ is said to be \emph{uniformly stable for the order $\leq_A$}
  if, given $\eps >0$, there is a $\delta>0$ such that, if $x$, $y\in K$ satisfy
  $\di(x,y)<\delta$ and $x\leq_A y$ or $y\leq_A x$, then
  $\di(u(t,\w,x),u(t,\w,y))< \eps$ for each $t\geq 0$.
\end{definition}
In order to obtain the 1-covering property of some omega-limit sets, in addition
to Hypotheses \ref{F1}--\ref{F4}, the componentwise separating property and the
uniform stability are assumed.
\begin{enumerate}[resume*=infdelay_properties]\setlength\itemsep{0.4em}
\item\label{F5} If $(\w,x)$, $(\w,y)\in\Om\times BC$ admit a backward orbit
  extension, $x\leq_A y$, and there is a subset $J\subset\{1,\ldots,m\}$ such
  that
  \begin{align*}
    x_i=y_i & \quad \text{ for each } i\notin J\,,\\
    x_i(s)< y_i(s) & \quad \text{ for each } i\in J\;\text{ and }
                     s\leq 0\,,
  \end{align*}
  then $F_i(\w,y)-F_i(\w,x) - (A\,(y(0)-x(0)))_i> 0$ for each $i\in J$.
\item\label{F6} For each $k\in\N$, $B_k$ is uniformly stable for the order
  $\leq_A$.\smallskip
\end{enumerate}
\par
The following result follows from Theorem~5.6 of~\cite{paper:NOV}.
\begin{theorem}\label{copiaLipschitz}
  Under assumptions~\ref{F1}--\ref{F6}, we consider the monotone skew-product
  semiflow~\eqref{skewBU} induced by~\eqref{infdelay}.  Fix
  $(\w_0,x_0)\in \Om\times BC$ such that $x_0$ is Lipschitz continuous and
  $z(\cdot,\w_0,x_0)$ is bounded for the norm $\n{{\cdot}}_\infty$.  Then
  $\mathcal{O}(\w_0,x_0)=\{(\w,c(\w))\mid \w\in\Om\}$ is a copy of the base and
  \[\lim_{t\to\infty}
    \textup\di(u(t,\w_0,x_0),c(\w_0{\cdot}t))=0\,,\] where $c:\Om\to BU$ is a
  continuous equilibrium, i.e. $u(t,\w,c(\w))=c(\w{\cdot}t)$ for each $\w\in\Om$
  and $t\geq 0$.
\end{theorem}
The aim of the rest of this section is to extend the previous characterization
to a more general class of initial data, not necessarily Lipschitz
continuous. More precisely, the functions $x$ of $BC$ satisfying the following
property:
\begin{enumerate}[label={\upshape({\bf R})},leftmargin=23pt]
\item\label{functional:regularity} $x$ is of bounded variation componentwise on
  $[-k,-k+1]$ for all $k\in\N$ and
  \begin{equation*}
    \sup\left\{V_{[-k,-k+1]}(x_i) \mid i\in\{1,\ldots,m\},\,k\geq
      1\right\}<\infty\,,
  \end{equation*}
\end{enumerate}
where $V_{[a,b]}(f)$ denotes the total variation of the scalar function $f\colon [a,b]\to \R$ on the interval $[a,b]$. \par
Note that the subset $\mathcal{R}$ of all the functions in $BC$ satisfying
property~\ref{functional:regularity} is a vector subspace of $BC$. Moreover,
$\mathcal{R}$ is a Banach space when endowed with the norm defined for
$x \in \mathcal{R}$ by
\[
  \n{x}_R=\n{x}_\infty + \sup\left\{V_{[-k,-k+1]}(x_i) \mid
    i\in\{1,\ldots,m\},\,k\geq 1\right\}\,.
\]
\par %%%%%%%%%%%%%%%%%%%%%%%%%%%%%%%%%%%%
Let us prove a useful characterization of this property in terms of the
existence of a common upper bound of $x$ and 0 for the exponential ordering
$\leq_A$. It is noteworthy that property~\ref{functional:regularity} does not
depend on the choice of the quasipositive matrix $A$.  We will denote by
$e^{a{\cdot}}$ the function $(-\infty,0]\to\R$, $t\mapsto e^{a t}$ for each
$a\in\R$.
\begin{proposition}\label{functional:characterization}
  Let $x\in BC$. The following statements are equivalent:
  \begin{enumerate}[label={\upshape(\roman*)},series=char]
  \item\label{functional:char_regular} $x$ satisfies
    property~\ref{functional:regularity};
  \item\label{functional:char_BC} there exists $h\in BC$ such that $h\geq_A x$
    and $h\geq_A 0$.
  \end{enumerate}
\end{proposition}
\begin{proof}
  Since $A$ is a diagonal matrix, we may assume without loss of generality that
  we are dealing with a scalar problem, i.e. $m=1$ and $A=(-a)$ for some
  $a>0$. \par
  (i) $\Rightarrow$ (ii) Let $c=\sup_{k\geq 1} V_{[-k,-k+1]}(x)$. We fix
  $t\in(-\infty,0]$ and let $\lceil t \rceil$ denote the integer part of the
  negative real number $t$, i.e.  $\lceil t \rceil-1< t\leq \lceil t \rceil$.
  Then, taking into account the properties of the bounded variation of the
  product of two functions and the increasing character of $e^{a{\cdot}}$, we
  deduce that
  \begin{align*}
    V_{(-\infty,t]}(e^{a{\cdot}}\,x) & \leq V_{(-\infty,\lceil t \rceil]}(e^{a{\cdot}}x)\leq \sum_{j=-\lceil t \rceil +1}^\infty V_{[-j,-j+1]}(e^{a{\cdot}}x) \\
                                     & \leq \sum_{j=-\lceil t \rceil +1}^\infty \left[ e^{a\,(-j+1)} V_{[-j,-j+1]}(x) +(e^{a\,(-j+1)}- e^{-a\,j}) \n{x}_\infty\right]\\
                                     &\leq \left[ c\,e^a+\n{x}_\infty(e^a-1)\right] \sum_{j=-\lceil t \rceil +1}^\infty e^{-a j}=C \, e^{a\,(\lceil t \rceil-1)}\leq C\,e^{a\,t}\,,
  \end{align*}
  where $C=(c \,e^a+\n{x}_\infty(e^a-1))/(1-e^{-a})$. This proves that
  $e^{a{\cdot}}\,x$ is a function of bounded variation on $(-\infty,0]$ and we
  can define $h$ as follows:
  \begin{equation*}
    \begin{array}[t]{cccl}
      h \colon& (-\infty,0]&\longrightarrow&\R\\
              & t&\mapsto& e^{-a\,t}\,V_{(-\infty,t]}(e^{a{\cdot}}\,x),
    \end{array}
  \end{equation*}
  which is clearly bounded by $C$. The continuity of $h$ follows from that of
  $e^{a{\cdot}}x$ (see Ex.~4 on p.~137 of Cohn~\cite{book:cohn}).  Moreover,
  from Proposition~4.4.2 of~\cite{{book:cohn}} the functions
  $e^{a \cdot}\,h:t \mapsto V_{(-\infty,t]}(e^{a{\cdot}}\,x)$ and
  $e^{a\cdot } (h-x):t \mapsto V_{(-\infty,t]}(e^{a{\cdot}}\,x)- e^{at}x(t)$ are
  nonnegative and nondecreasing, which implies that $h\geq_Ax$ and $h\geq_A0$
  and (ii) holds.
  \par
  (ii) $\Rightarrow$ (i) Since $A=(-a)$, from $h\geq_A 0$ and $h\geq_A x$, we
  deduce that $e^{a \cdot } \,h$ and $e^{a \cdot }\,(h-x)$ are nonnegative and
  nondecreasing. Consequently, the function
  $e^{a \cdot } \,x= e^{a \cdot } \,h- e^{a \cdot }\,(h-x)$ is of bounded
  variation on $(-\infty,0]$ and hence on $[-k,-k+1]$ for each
  $k\in\N$. Moreover,
  \begin{multline*}
    V_{[-k,-k+1]}(e^{a{\cdot}} x) \leq V_{[-k,-k+1]}(e^{a{\cdot}} h)+V_{[-k,-k+1]}(e^{a{\cdot}} (h-x))\\
   \leq e^{a\,(-k+1)}h(-k+1) + e^{a\,(-k+1)}(h(-k+1)-x(-k+1))\leq D\,
    e^{-a\,k}\,,
  \end{multline*}
  where $D=e^a(2\n{h}_\infty+\n{x}_\infty)$. In addition, $e^{-a\cdot}$ is also
  of bounded variation on $[-k,-k+1]$, whence we deduce the same for
  $x=e^{-a \cdot }e^{a \cdot } x$. Therefore,
  \begin{equation*}
    V_{[-k,-k+1]}(x)\leq e^{ak}\,D\,e^{-a\,k}+ e^{a\,k}e^{a\,(-k+1)}\n{x}_\infty=D+e^a\n{x}_\infty\,,
  \end{equation*}
  which is a bound irrespective of $k$ and (i) holds.
\end{proof}
\begin{remark}
  Notice that all Lipschitz continuous initial data satisfy
  property~\ref{functional:regularity}, but the converse does not
  hold. Actually, there exist functions in $BC$ satisfying
  property~\ref{functional:regularity} and which are not in $BU$.
\end{remark}
A special choice of the function $h$ of
Proposition~\ref{functional:characterization} will be important for later
purposes.
\begin{lemma}\label{h0}
  Fix $x_0\in BC$ satisfying property~\ref{functional:regularity}.  Then there
  exists $h_0\in BC$ which, in addition to $h_0\geq_A x_0$ and $h_0\geq_A 0$,
  satisfies
  \begin{equation*}
    h_0\geq _A x_0-x_0(0)\,,
  \end{equation*}
  where $x_0(0)$ represents the constant function with that value.
\end{lemma}
\begin{proof}
  If $\wt h_0\in BC$ is the function given in
  Proposition~\ref{functional:characterization}(ii), then the function
  $h_0=\wt h_0+\n{x_0}_\infty$, where $\n{x_0}_\infty$ represents the constant
  function from $(-\infty,0]$ into $\R^m$ with that value in all the components,
  satisfies the desired properties.
\end{proof}
The main theorem of the section provides the 1-covering property of omega-limit
sets when the initial data $x_0$ are in $BC$ and satisfies
property~\ref{functional:regularity}.
\begin{theorem}\label{copiabasenuevo}
  Let $(\w_0,x_0)\in \Om\times BC$. Under assumptions~\ref{F1}--\ref{F6}, if
  $x_0$ satisfies property~\ref{functional:regularity} and $z(\cdot,\w_0,x_0)$
  is a bounded solution of~\eqref{infdelay}$_{\w_0}$, the omega-limit set
  $\mathcal{O}(\w_0,x_0)=\{(\w,c(\w))\mid \w\in\Om\}$ is a copy of the base and
  \[\lim_{t\to\infty}
    \textup\di(u(t,\w_0,x_0),c(\w_0{\cdot}t))=0\,,\] where $c:\Om\to BU$ is a
  continuous equilibrium, i.e. $u(t,\w,c(\w))=c(\w{\cdot}t)$ for each $\w\in\Om$
  and $t\geq 0$, and it is continuous for the compact-open topology on $BU$.
\end{theorem}
\begin{proof}
  Fix a point $(\wt \w_0,\wt x_0)\in \mathcal{O}(\w_0,x_0)$. As explained
  before, every point in the omega-limit set $\mathcal{O}(\w_0,x_0)$ admits a
  unique backward orbit which remains inside the set
  $\mathcal{O}(\w_0,x_0)$. From this fact, we deduce that
  $u(t,\wt \w_0,\wt x_0)$ is defined for each $t\in\R$ and $\wt x_0$ is
  continuously differentiable because $\wt x_0(t)= z(t,\wt \w_0,\wt x_0)$ for
  each $t\in (-\infty,0]$. Therefore, from~\eqref{infdelay}$_{\wt \w_0}$
  and~\eqref{F2}, the Lipschitz character of $\wt x_0$ is deduced. An
  application of Theorem~\ref{copiaLipschitz} yields
  $\mathcal{O}(\wt\w_0,\wt x_0)=\{(\w,c(\w))\mid \w\in\Om\}$ for a continuous
  equilibrium $c:\Om\to BU$. We claim that
  $\mathcal{O}(\wt\w_0,\wt x_0)=\mathcal{O}(\w_0, x_0)$, which will finish the
  proof.  We know that
  $\mathcal{O}(\wt\w_0,\wt x_0) \subset \mathcal{O}(\w_0, x_0)$; to check the
  coincidence of both sets, it is enough to prove that, given $\eps >0$, there
  is a $T>0$ such that
  \[\di(u(t,\w_0,x_0),c(\w_0{\cdot}t))<\eps\quad
    \text{ for each } t> T \,.\]
  \par
  Since $(\w_0,c(\w_0))\in \mathcal{O}(\w_0, x_0)$, there exists a sequence
  $t_n\uparrow \infty$ such that
  \[\w_0=\lim_{n\to\infty} \w_0{\cdot}t_n
    \quad \text{ and } \quad c(\w_0)\stackrel{\di\;}=\lim_{n\to\infty}
    u(t_n,\w_0,x_0)\,.\] However, since
  $u(t_n,\w_0,c(\w_0))=c(\w_0{\cdot}t_n) \stackrel{\di\;}\to c(\w_0)$ as
  $n\uparrow \infty$, we deduce that
  \begin{equation}\label{limittn}
    \lim_{n\to\infty}\di(u(t_n,\w_0,x_0),c(\w_0{\cdot}t_n))=0\,.
  \end{equation}
  Next, we will approximate $u(t,\w_0,x_0)$ by the Lipschitz continuous function
  of $BU$ $v(t,\w_0,x_0)$ defined by
  \begin{equation}\label{vx0r}
    v(t,\w_0,x_0)(s)=\begin{cases} u(t,\w_0,x_0)(s) & \text{if } s\in[-t,0]\,,\\
      x_0(0)           & \text{if } s\in(-\infty,-t]\,. \\
    \end{cases}
  \end{equation}
  It is immediate to see that
  \begin{equation}\label{vx0r-ux0r}
    \lim_{t\to\infty} \di\big(u(t,\w_0,x_0),v(t,\w_0,x_0)\big)=0\,.
  \end{equation}
  We introduce the following auxiliary continuous functions defined from the
  function $h_0$ provided by Lemma~\ref{h0}:
  \begin{align*}
    &\overline{h}\colon \R \to \R^m\,,\qquad \qquad\;  s \mapsto  \begin{cases}
      e^{As}\,h_0(0) & \text{ if } s>0\,, \\
      h_0(s)       & \text{ if } s\leq 0\,,
    \end{cases} \\[.1cm]
    &\overline{h}_T \colon (-\infty,0]\to \R^m\,, \quad s\mapsto
      \overline{h}(s+T)\,, \quad T\in\R\,.
  \end{align*}
  It is easy to check that
  \begin{equation}\label{propiedades-hT}
    \overline{h}_T\geq_A 0\; \text{ for each }T\in \R  \quad \text{and}\quad \overline{h}_T \stackrel{\di\;}\to 0\; \text{ as }\; T\to \infty\,.
  \end{equation}
  From~\eqref{limittn}, \eqref{vx0r-ux0r} and \eqref{propiedades-hT}, it follows
  that
  \begin{align}
    &\lim_{n\to\infty} \di(u(t_n,\w_0,x_0),c(\w_0{\cdot}t_n))=0\,,\label{limitesTn1}\\[2pt]
    & \lim_{n\to\infty} \di\big(u(t_n,\w_0,x_0),v(t_n,\w_0, x_0)\big)=0\,, \label{limitesTn2}\\[2pt]
    & \lim_{n\to\infty} \di\big(\overline{h}_{t_n},0\big)=0\,, \label{limitesTn3}
  \end{align}
  Therefore, denoting for simplicity by $c_n$ and $v_n$ the functions of $BU$
  \[ c_n=c(\w_0{\cdot}t_n) \quad \text{and} \quad v_n=v(t_n,\w_0, x_0)\,,\]
  from~\eqref{limitesTn1} and~\eqref{limitesTn2}, it follows that
  \begin{equation}\label{limitesTn4}
    \lim_{n\to\infty}\di(c_n,v_n)=0\,.
  \end{equation}
  Next, as in Proposition~4.4 of~\cite{paper:NOV}, we define the functions
  $a_{v_n,c_n},\, b_{v_n,c_n}$ of $BU$ by
  \begin{equation*}
    \begin{array}{ccl}
      a_{v_n,c_n}\colon(-\infty,0]&\longrightarrow &\R^m\\
      \quad\;\; s &\mapsto& {\displaystyle \int_{-\infty}^s
                            e^{A(s-\tau)}\inf\{v_n'(\tau)-A\,v_n(\tau), c_n'(\tau)-A \,c_n(\tau)\}\,
                            d\tau \,,}
    \end{array}
  \end{equation*}
  \begin{equation*}
    \begin{array}{ccl}
      b_{v_n,c_n}\colon(-\infty,0]&\longrightarrow &\R^m\\
      \quad\;\; s &\mapsto& {\displaystyle \int_{-\infty}^s
                            e^{A(s-\tau)}\sup\{v_n'(\tau)-A\,v_n(\tau), c_n'(\tau)-A \,c_n(\tau)\}\,
                            d\tau \,,}
    \end{array}
  \end{equation*}
  which satisfy $a_{v_n,c_n}\leq_A v_n \leq_A b_{v_n,c_n}$ and
  $a_{v_n,c_n}\leq_A c_n \leq_A b_{v_n,c_n}$. Moreover, for each $s\le 0$, we
  have $\n{b_{v_n,c_n}(s)-c_n(s)}\leq \n{a_{v_n,v_n}(s)-b_{v_n,c_n}(s)}$, whence
  \begin{equation*}
    \n{b_{v_n,c_n}(s)-c_n(s)}\leq \int_{-\infty}^s\big\|e^{A (s-\tau)}\big\|
    \big(\n{v_n'(\tau)-c_n'(\tau)}+\n{A}\,\n{v_n(\tau)-c_n(\tau)}\big) \,d\tau\,.
  \end{equation*}
  Thanks to the definition of $v_n$ and~\eqref{vx0r},
  \begin{equation}\label{derivativesvncn}
    \begin{split}
      &v'_n(\tau)=\begin{cases} F(\w_0{\cdot}(t_n+\tau),u(t_n+\tau,\w_0,x_0)) &\text{ if } \tau\in(-t_n,0]\,,\\
        0 & \text{ if } \tau\in(-\infty,-t_n)\,,
      \end{cases}\\[2pt]
      &c'_n(\tau)=F(\w_0{\cdot}(t_n+\tau),c(\w_0{\cdot}(t_n+\tau)) \qquad \text{
        for each }\tau\leq 0\,,
    \end{split}
  \end{equation}
  whence~\eqref{F2} provides the uniform boundedness of $v_n$, $v'_n$, $c_n$ and
  $c'_n$ for all $n\in\N$. Now, for each $s\le 0$,
  $\n{e^{A\,(s-\tau)}}\leq e^{-a(s-\tau)}$ for some positive $a>0$ and
  $\int_{-\infty}^s e^{-a(s-\tau)}\,d\tau=1/a$. As a result, we deduce the
  existence of $T_0>0$ such that
  \[
    \int_{-\infty}^{-T_0}\big\|e^{A (s-\tau)}\big\|\,
    \big(\n{v_n'(\tau)-c_n'(\tau)}+\n{A}\,\n{v_n(\tau)-c_n(\tau)}\big) \,d\tau
    <\eps
  \]
  for each $n\in\N$ and, hence,
  \[
    \n{b_{v_n,c_n}(s)-c_n(s)}\leq \eps+\int_{-T_0}^0
    \big(\n{v_n'(\tau)-c_n'(\tau)}+\n{A}\,\n{v_n(\tau)-c_n(\tau)}\big)
    \,d\tau\,.
  \]
  Since $\lim_{n\to\infty}\di(v_n,c_n)=0$, the second part of the integral
  vanishes as $n\uparrow \infty$.  In addition, if we fix $n_0\in\N$ with
  $t_{n_0}>T_0$, the inequality $-t_n\leq-t_{n_0}\leq -T_0\le 0$ holds for each
  $n\geq n_0$ and, thanks to~\eqref{derivativesvncn}, \eqref{limitesTn4},
  \ref{F3} and the relative compactness of the trajectory
  $\{\tau(t,\w_0,x_0):t\ge 0\}$,
  \begin{multline*}
    \int_{-T_0}^0 \n{v_n'(\tau)-c_n'(\tau)}\,d\tau=\int_{-T_0}^0
    \big\|F(\w_0{\cdot}(t_n+\tau), u(t_n+\tau,\w_0,x_0))\\-
    F(\w_0{\cdot}(t_n+\tau),c(\w_0{\cdot}(t_n+\tau))\big\|\,d\tau\,\,
  \end{multline*}
  also tends to $0$ as $n\uparrow \infty$. Thus,
  $\lim_{n\to\infty}\di(b_{v_n,c_n},c_n)=0$ and, consequently,
  $\lim_{n\to\infty}\di(b_{v_n,c_n},v_n)=0$. Next, we consider the function
  $g_n=b_{v_n,c_n}+\overline{h}_{t_n}$, which satisfies
  $\lim_{n\to\infty}\di(g_n,b_{v_n,c_n})=0$ thanks
  to~\eqref{limitesTn3}. Therefore,
  \begin{equation}\label{gn-vn-cn}
    \lim_{n\to\infty}\di(g_n,c_n)=0 \quad \text{and}\quad  \lim_{n\to\infty}\di(g_n,v_n)=0\,.
  \end{equation}
  From Lemma~\ref{h0}, we know that $h_0\geq _A x_0-x_0(0)$ and $h_0\geq_A 0$,
  whence
  \[
    u(t_n,\w_0,x_0)\leq_A v(t_n,\w_0,x_0)+
    \overline{h}_{t_n}=v_n+\overline{h}_{t_n}\,,
  \]
  i.e. the function
  $v:s \mapsto e^{-A\,s}\big(v_n(s)+ \overline{h}_{t_n}(s)
  -u(t_n,\w_0,x_0)(s)\big)$ is nondecreasing because we can write
  \[
    v(s)=\begin{cases} e^{-A\,s}\,\overline{h}_{t_n}(s) &
      \text{ if } s\in[-t_n,0]\,, \\[2pt]
      e^{A \,t_n}\, e^{-A(s+t_n)} \big[\,x_0(0)-x_0(s+t_n)+h_0(s+t_n) \big]&
      \text{ if } s\in(-\infty,-t_n] \,.\end{cases}
  \]
  Therefore, $u(t_n,\w_0,x_0)\leq_A g_n$; hence,~\eqref{limitesTn2}
  and~\eqref{gn-vn-cn} yield
  \begin{equation}\label{uTngnto0}
    \di(u(t_n,\w_0,x_0),g_n)\leq \di(u(t_n,\w_0,x_0),v_n)+\di(v_n,g_n)\underset{n\to\infty}\longrightarrow 0\,.
  \end{equation}
  Thanks to the uniform boundedness of $g_n$ and $c_n$, there exist a constant
  $c>0$ such that $u(t_n,\w_0,x_0)$, $c(\w_0{\cdot}t_n)$ and $g_n\in B_c$ for
  each $n\in\N$. Let $\delta>0$ be the modulus of uniform stability of $B_c$ for
  the order $\leq_A$ for $\eps/2$.  From~\eqref{uTngnto0} and~\eqref{gn-vn-cn},
  it follows that there exists $n_0\in\N$ such that
  \[\di(u(t_n,\w_0,x_0),g_n)<\delta
    \quad \text{and} \quad \di(c(\w_0{\cdot}t_n),g_n)<\delta\] for each
  $n\geq n_0$. Therefore, from $u(t_n,\w_0,x_0)\leq_ A g_n$,
  $c(\w_0{\cdot}t_n)=c_n\leq_A g_n$ and~\ref{F6}, we deduce that
  \begin{align*}
    &\di\big(u(t,\w_0{\cdot}t_n,u(t_n,\w_0,x_0)),u(t,\w_0{\cdot}t_n,g_n)\big)<\frac{\eps}{2}\,,\\
    &\di\big(u(t,\w_0{\cdot}t_n,c(\w_0{\cdot}t_n)), u(t,\w_0{\cdot}t_n,g_n)\big)<\frac{\eps}{2}\,,
  \end{align*}
  that is,
  \[
    \di\big(u(t+t_n,\w_0,x_0), c(\w_0{\cdot}(t+t_n)\big)<\eps\quad \text{for
      each } t\geq 0 \text{ and } n\geq n_0\,.
  \]
  Finally, taking $T=t_{n_0}$ yields the expected result.
\end{proof}
\section{Transformed exponential order for NFDEs with infinite
  delay}\label{neutral}
In this section we will extend the previous results to NFDEs with infinite delay
and nonautonomous stable operator.  This extension requires the definition
of a transformed exponential ordering.
\subsection{Nonautonomous stable operators}\label{operatorD}
Let $D\colon\Om\times BC\to\R^m$ be an operator satisfying the following
hypotheses:
\begin{enumerate}[label={\upshape(\textbf{D\arabic*})},series=D_properties,
  leftmargin=27pt]\setlength\itemsep{0.4em}
\item\label{neutral:D_linear} $D$ is linear and continuous in its second
  variable for the norm $\n{\cdot}_\infty$ and the map
  $\Om\to\mathcal{L}(BC,\R^m)$, $\w\mapsto D(\w,{\cdot})$ is continuous;
\item\label{neutral:D_cont_d} for each $r>0$, $D\colon\Om\times B_r\to \R^m$ is
  continuous when we take the restriction of the compact-open topology to $B_r$,
  i.e.~if $\w_n\to\w$ and $x_n\stackrel{\di\;}\to x$ as $n\uparrow\infty$ with
  $x\in B_r$, then $\lim_{n\to\infty}D(\w_n,x_n)=D(\w,x)$.
\end{enumerate}
Under these assumptions, by adapting the proof of Lemma~3.1 of~\cite{paper:obvi}
for each $\w\in\Om$, we deduce that
\[ D(\w,x)=\int_{-\infty}^0 [d\mu(\w)(s)]\,x(s)=B(\w)\,x(0)-\int_{-\infty}^0
  [d\nu(\w)(s)]\,x(s)\,,\] where $\mu(\w)=[\mu_{ij}(\w)]$, $\mu_{ij}(\w)$ is a
real regular Borel measure with finite total variation
$|\mu_{ij}(\w)|((-\infty,0])<\infty$, for $i$, $j\in\{1,\ldots,m\}$ and
$\w\in\Om$, $B(\w)=\mu(\w)(\{0\})$ and $\nu(\w)=B(\w)\,\delta_0-\mu(\w)$ where
$\delta_0$ is the Dirac measure at 0. As in Corollary 3.4 of~\cite{paper:obvi},
it follows that
\begin{equation}\label{eq:nu}
  \lim_{\rho\to 0^+}\n{\nu(\w)}_\infty([-\rho,0])=0 \text{ and } \lim_{\rho\to \infty}\n{\nu(\w)}_\infty((-\infty,-\rho])=0
\end{equation}
uniformly for $\w\in\Om$, where $\n{\nu(\w)}_\infty(E)$ denotes the matrix norm
associated to the maximum norm of the $m\times m$ matrix $[|\nu_{ij}(\w)|(E)]$
of total variations over the Borel subset $E\subset(-\infty,0]$, i.e.
$ \displaystyle \n{\nu(\w)}_\infty(E)=\max_{1\leq i\leq
  m}\sum_{j=1}^{m}|\nu_{ij}(\w)|(E)\,.$\par
In addition, we will assume a natural generalization of the atomic character at
$0$ of the operator $D$, as seen in Hale~\cite{book:hale}, Hale and
Verduyn-Lunel~\cite{book:hale2}, and~\cite{paper:obvi}.
\begin{enumerate}[resume*=D_properties]
\item\label{neutral:D_atomic} $B(\w)$ is the identity matrix for all $\w\in\Om$.
\end{enumerate}
Thus, under assumptions~\ref{neutral:D_linear}--\ref{neutral:D_atomic}, $D$
takes the form
\begin{equation*}
  D(\w,x)=x(0)-\int_{-\infty}^0 [d\nu(\w)(s)]\,x(s)\,,
\end{equation*}
where $\nu$ has a continuous variation with respect to $\omega$ and
satisfies~\eqref{eq:nu}.
\par
We omit the proof of the following result, which can be easily adapted to this
case from Theorem~2.5 of~Mu\~{n}oz-Villarragut~\cite{tesis:victor}.  As stated before, given a continuous function $x\in C(\R,\R^m)$,  $x_t(\cdot)$ denotes the continuous function $x_t\colon (-\infty,0]\to \R^m$  defined  by $x_t(s)=x(t+s)$
for $s\in (-\infty,0]$.
\begin{proposition}\label{existencia}
  Under assumptions~\ref{neutral:D_linear}--\ref{neutral:D_atomic}, for each
  $h\in C([0,\infty),\R^m)$ and $(\w,\varphi)\in \Om\times BC$ with
  $D(\w,\varphi)=h(0)$, the nonhomogeneous equation
  \begin{equation}\label{nohomogenea}
    \left\{
      \begin{array}{ll}
        D(\w{\cdot}t,x_t)=h(t)\,, & t\geq 0\,, \\
        x_0=\varphi\,,
      \end{array} \right.
  \end{equation}
  has a solution $x\in C(\R,\R^m)$.
\end{proposition}
Next the definition of stability for $D$ is stated.
\begin{enumerate}[resume*=D_properties]
\item\label{neutral:D_stable} $D$ is \emph{stable}, that is, there is a
  continuous function $c\in C([0,\infty),\R)$ with $\lim_{\,t\to\infty}c(t)=0$
  such that, for each $(\w,\varphi)\in\Om\times BC$ with $D(\w,\varphi)=0$, the
  solution of the homogeneous problem
  \[\left\{
      \begin{array}{ll}
        D(\w{\cdot}t,x_t)=0\,, & t\geq 0\,, \\
        x_0=\varphi\,,
      \end{array} \right.
  \] satisfies $\n{x(t)}\leq c(t)\,\n{\varphi}_\infty$ for each $t\geq 0$.
\end{enumerate}
\par
The next result provides a condition to check the stability of $D$. Its
proof is similar to the one for $BU$ done in
Theorem~3.9(iii) of~\cite{paper:obvi}. We will denote
$\wh D_2(\w,x)\colon(-\infty,0]\to\R^m$, $s\mapsto D(\w{\cdot}s,x_s)$.
\par
\begin{proposition}\label{stabilityCondition} Under
  assumptions~\ref{neutral:D_linear}--\ref{neutral:D_atomic}, if for each $r>0$
  and each sequence $\{(\w_n,x^n)\}_n\subset\Om\times BC$ such that
  $\n{\widehat D_2 (\w_n,x^n)}_\infty\leq r$, $\w_n\to\w\in\Om$ and
  {\upshape$\widehat D_2(\w_n,x^n) \stackrel{\di}\to 0$} as $n\uparrow\infty$,
  it holds that $x^n(0)\to 0$ as $n\to\infty$, then $D$ is stable.
\end{proposition}
Although the definition of stability is given for the homogeneous equation, it
is easy to deduce quantitative estimates for the solution of a non-homogeneous
equation in terms of the initial data. The proof of the next proposition  is analogous to the one for $BU$ done in Theorem~2.11 and
Proposition~3.2 of~\cite{tesis:victor}.
\begin{proposition}\label{prop:bounds}
  Under assumptions~\ref{neutral:D_linear}--\ref{neutral:D_stable}, there are a
  positive constant $k>0$ and a continuous function $c\in C([0,\infty),\R)$ with
  $\lim_{\,t\to\infty}c(t)=0$ such that
  \begin{itemize}
  \item[\rm{(i)}]
    $\n{x^h(s)}\leq c(t)\,\n{x^h}_\infty+ k\,{\displaystyle\sup_{s-t\leq \tilde
        s\leq s}\n{h(\tilde s)}}\,$ for all $s\leq 0\leq t$, and hence
  \item[\rm{(ii)}] $\n{x^h}_\infty\leq k\,\n{h}_\infty$\,,
  \end{itemize}
for each $h\in BC,\,\w\in\Om$ and
  $x^h\in BC$ satisfying $D(\w{\cdot}s,x_s^h)=h(s)$ for $s\leq 0$.
\end{proposition}
The following statement associates the stability of $D$ to the invertibility of its convolution operator
$\widehat{D}$ and will allow us to transform the family of NFDEs with infinite delay~\eqref{neutral:infdelay} and nonoautonomous stable operator $D$ into a family of FDEs with infinite delay. We refer to Staffans~\cite{staf83} for the case of autonomous stable operators $D$ in an appropriate fading memory space, to~Haddock {\it et al.}~\cite{HKW1990} for an application of these ideas, and to~\cite{paper:MNO} (resp.~\cite{paper:obvi}) for the cases of autonomous (resp. nonautonomous) stable operators $D$ in $BU$.
\begin{theorem}\label{neutral:Dhat_properties}
  Under assumptions~\ref{neutral:D_linear}--\ref{neutral:D_stable}, we define
  the map
  \[
    \begin{array}{lcclrcl}
      \wh{D}\colon &\Om\times BC&\longrightarrow & \Om\times BC &&\\
                   & (\w,x) & \mapsto &(\w,\wh{D}_2(\w,x))
    \end{array}
  \]
  where $\wh D_2(\w,x)\colon(-\infty,0]\to\R^m$, $s\mapsto
  D(\w{\cdot}s,x_s)$. Then
  \begin{itemize}
  \item[\rm{(i)}] $\wh D$ is well defined and invertible;
  \item[\rm{(ii)}] $\wh D_2$ and $(\wh{D}^{-1})_2$ are linear and continuous for
    the norm in their second variable for all $\w\in\Om$; and
  \item[\rm{(iii)}] for all $r>0$, $\wh D$ and $\wh{D}^{-1}$ are uniformly
    continuous on $\Om\times B_r$ when we take the restriction of the
    compact-open topology to $B_r$.
  \end{itemize}
\end{theorem}
\begin{proof} (i) We check that $\wh D_2(\w,x)\in BC$. The continuity follows
  from~\ref{neutral:D_cont_d}, and the boundedness from~\ref{neutral:D_linear}
  because
  \[\n{\wh D_2(\w,x)}_\infty=\sup_{s\leq 0} \n{D(\w{\cdot}s,x_s)}\leq \sup_{\wt
      \w\in\Om}\n{D(\wt \w,\cdot)}\,\n{x}_\infty<\infty\,.\] $\widehat{D}$ is
  injective because, if we have $(\w,x)$, $(\wh \w,\wh x)\in\Om\times BU$ with
  $\widehat D(\w,x)=\widehat D(\wh \w,\wh x)$, then $\w=\wh \w$ and, from
  Proposition~\ref{prop:bounds}(ii) and the fact that
  $D(\w{\cdot}s,x_s-\wh x_s)=0$ for $s\leq 0$, we get $x=\wh x$.\par
  In order to show that $\wh D$ is surjective, let $(\w,h)\in\Om\times BC$. As
  in Theorem~3.9 of~\cite{paper:obvi}, we take a sequence of continuous
  functions with compact support $\{h_n\}_n\subset B_r$ for some $r>0$ such that
  $h_n\stackrel{\di\;}\to h$ as $n\uparrow\infty$, and a sequence $\{x^n\}_n$ of
  continuous functions with compact support such that $\wh D_2(\w,x^n)=h_n$,
  i.e. $D(\w{\cdot}s,x_s^n)=h_n(s)$ for each $s\leq 0$. The next step of the
  proof differs from the one in Theorem~3.9 of~\cite{paper:obvi} because now $h$
  belongs to $BC$ instead of~$BU$.  We will check that given $\rho>0 $ and
  $\eps >0$ there is a $\delta>0$ such that
  \begin{equation}\label{equicon}
    \n{x^n-x^n_\tau}_{[-\rho,0]} <\eps \quad \forall \,n\in\N \;\text{ and }\;  \tau\in[-\delta,0]\,.
  \end{equation}
  For each $\tau\leq 0$ we define
  \[g_n^\tau\colon (-\infty,0]\to \R^m\,,\; s\mapsto D(\w{\cdot}s,
    (x^n-x^n_\tau)_s)\,,\] that is, $D(\w{\cdot}s,(x^n-x^n_\tau)_s)=g_n^\tau(s)$
  for each $s\leq 0$, and from Proposition~\ref{prop:bounds}(i) we deduce that
  \[ \n{x^n(s)-x^n_\tau(s)}\leq c(t)\,\n{x^n-x_\tau^n}_\infty+ k\, \sup_{s-t\leq
      \wt s\leq s} \n{g_n^\tau(\tilde s)} \quad \text {for each } t\geq
    0,\;s\leq 0\,.\] Since $c(t)\to 0$ as $t\uparrow \infty$ and from
  Proposition~\ref{prop:bounds}(ii) we have $\n{x^n_\tau}_\infty\leq k\,r$ for
  each $n\in\N$ and $\tau\leq 0$, we can find a $T>0$ such that
  $c(T)\, \n{x^n-x_\tau^n}_\infty<\eps/2$ and if $s\in[-\rho,0]$ then
  \begin{equation}\label{desi:xn}
    \n{x^n-x_\tau^n}_{[-\rho,0]}\leq \frac{\eps}{2}+k\, \n{g_n^\tau}_{[-\tilde\rho,0]}\,.
  \end{equation}
  for $\tilde \rho=\rho+T$. From the equicontinuity of $\{h_n\}_n$
  and~\ref{neutral:D_cont_d} there is a $\delta>0$ such that for each
  $\tau\in[-\delta,0]$, $s\in[-\tilde \rho,0]$ and $n\in\N$
  \[\n{h_n-(h_n)_\tau}_{[-\tilde \rho,0]}<\frac{\eps}{4\,k}\;\; \text{ and }
    \;\;
    \n{D(\w{\cdot}(s+\tau),\cdot)-D(\w{\cdot}s,\cdot)}<\frac{\eps}{4\,k^2\,r}\,.\]
  Thus, from the definitions of $g_n^\tau$ and $h_n$ we deduce that
  \begin{align*}
    \n{g_n^\tau(s)}&\leq \n{h_n(s)-(h_n)_\tau(s)} +\n{D(\w{\cdot}(s+\tau),x^n_{s+\tau})-D(\w{\cdot}s,x^n_{s+\tau})}\\
                   & \leq \frac{\eps}{4\,k}+ \frac{\eps}{4\,k^2\,r}\n{x^n}_\infty \leq \frac{\eps}{2\,k}
  \end{align*}
  for each $s\in[-\tilde \rho,0]$, which together with~\eqref{desi:xn}
  yields~\eqref{equicon}, as stated. Thus, $\{x^n\}_n$ is equicontinuous and,
  consequently, relatively compact for the compact-open topology. Hence, there
  is a convergent subsequence of $\{x^n\}_n$ (let us assume it is
    the whole sequence), i.e. there is a continuous function $x$ such that
    $x^n\stackrel{\di}\to x$ as $n\to\infty$. Therefore, we have that
    $\n{x}_\infty\leq k\,r$, which implies that $x\in BC$. From this,
    $x^n_s \stackrel{\textsf{d}\;}\to x_s$ for each $s\leq 0$ and the expression
    of $D$ yields $D(\w{\cdot}s,x_s^n)=h_n(s)\to D(\w{\cdot}s,x_s)$, i.e.
    $D(\w{\cdot}s,x_s)=h(s)$ for $s\leq 0$ and $\widehat D_2(\w,x)=h$. Then
    $\widehat{D}$ is surjective, as claimed.
    \par
    (ii) The continuity of $\wh D_2$ for the norm in the second variable is a
    consequence of~\ref{neutral:D_linear}, and the corresponding property for
    $(\wh{D}^{-1})_2$ follows from~Proposition~\ref{prop:bounds}(ii).
    \par
    (iii) The proof of the uniform continuity of $\wh D$ (resp.~$\wh{D}^{-1}$)
    on $\Om\times B_r$ for the compact-open topology is omitted because it
    follows, adapted to this case, the same steps of Theorem~3.6
    of~\cite{paper:obvi} (resp. Theorem~3.9 of~\cite{paper:obvi}).
  \end{proof}
\subsection{Neutral functional differential equations}
Let us consider the family of NFDEs with infinite delay
\begin{equation}\label{neutral:infdelay}
  \frac{d}{dt}D(\w{\cdot}t,z_t)=G(\w{\cdot}t,z_t),\quad t\geq 0,\,\w\in\Om,
\end{equation}
defined by an operator $D\colon\Om\times BC\to\R^m$ and a function
$G\colon\Omega\times BC \to\R^m$.
\par
With the notation of the previous subsection and a diagonal matrix $A$ with
negative diagonal entries as is Section~\ref{secfun}, we define the
\emph{transformed exponential order} relation introduced in~\cite{paper:obvi} on
each fiber of the product $\Om\times BC$: if $(\w,x),\,(\w,y)\in\Om\times BC$,
then
\begin{equation}\label{transfexporder}
  (\w,x)\leq_{D,A}(\w,y) \;\Longleftrightarrow\; \wh
  D_2(\w,x)\leq_A\wh D_2(\w,y),
\end{equation}
based on the partial order relation $\leq_A$ on $BC$ given in~\eqref{orderA}.
\par
Let us assume the following hypotheses:
\begin{enumerate}[label=\upshape(\textbf{N\arabic*}),series=neutral_properties,
  leftmargin=27pt]\setlength\itemsep{0.4em}
\item\label{N1} $G:\Om\times BC\to\R^m$ is continuous on $\Om\times BC$  when the norm $\|\cdot\|_\infty$ is considered on $BC$  and its
  restriction to $\Om\times B_r$ is Lipschitz continuous in its second variable  for each $r>0$; \smallskip
\item \label{N3} for each $r>0$, the restriction of $G$ to $\Om\times B_r$ is
  continuous when the compact-open topology is considered on $B_r$; \smallskip
\item \label{N4} if $(\w,x),\,(\w,y)\in\Om\times BC$ and
  $(\w,x)\leq_{D,A}(\w,y)$, then $G(\w,y)-G(\w,x)\geq A(D(\w,y)-D(\w,x))$ for the usual componentwise partial order relation on $\R^m$.
\end{enumerate}
Notice that assumption~\ref{N1} implies
\begin{equation*}
  \text{$G(\Om\times B_r)$ is a bounded subset of $\R^m$ for each $r>0$.}
\end{equation*}
As in the case of conditions~\ref{F1} and~\ref{F3}, condition~\ref{N1} does not imply~\ref{N3}. The reason for writing them separately is that~\ref{N1} provides existence and uniqueness in $BU$, while~\ref{N3} is used in the case of $BC$.\par
Under assumptions~\ref{neutral:D_linear}--\ref{neutral:D_stable} and~\ref{N1},
as seen in Wang and Wu~\cite{paper:WW1985} and~\cite{paper:jwu1991}, for each
$\w\in\Om$, the local existence and uniqueness of the solutions of
equation~\eqref{neutral:infdelay}$_\w$ is guaranteed if we assume some
hypotheses on the phase space that, in particular, are satisfied by
$BU$. Moreover, given $(\w,x)\in\Om\times BU$, if $z({\cdot},\w,x)$ represents
the solution of equation~\eqref{neutral:infdelay}$_\w$ with initial datum $x$,
then $u(t,\w,x):(-\infty,0]\to\R^m$, $s\mapsto z(t+s,\w,x)$ is an element of
$BU$ for all $t\geq 0$ where $z({\cdot},\w,x)$ is defined.\par
As a result, a local skew-product semiflow can be defined on $\Om\times BU$:
\[
  \begin{array}{rccl}
    \tau:&\mathcal U\subset\R^+\times\Om\times BU&\longrightarrow&\Om\times BU\\
         &(t,\w,x)&\mapsto&(\w{\cdot}t,u(t,\w,x)).
  \end{array}
\]
Next, let $(\w,y)\in\Om\times BU$. For each $t\geq 0$ such that
$u(t,\wh D^{-1}(\w,y))$ is defined, we define
$\wh u(t,\w,y)=\wh D_2(\w{\cdot}t,u(t,\wh D^{-1}(\w,y)))$.  As seen
in~\cite{paper:obvi}, it can be checked that
\[
  \wh z(t,\w,y)=\begin{cases}
    \,y(t)&\text{if }t\leq 0,\\
    \,\wh u(t,\w,y)(0)&\text{if }t\geq 0,
  \end{cases}
\]
is the solution of
\begin{equation}\label{neutral:transformed_family}
  \wh z\,'(t)=F(\w{\cdot}t,\wh z_t),\quad t\geq 0,\,\w\in\Om
\end{equation}
with initial datum $y$, where $F=G\circ\wh D^{-1}$.
\par\smallskip
A similar proof to that of Proposition~4.1 of~\cite{paper:obvi} provides the
following result.
\begin{proposition}\label{neutral:transformed_conditions1234}
  Under assumptions~\ref{neutral:D_linear}--\ref{neutral:D_stable} and
  \ref{N1}--\ref{N4}, the map $F=G\circ\wh D^{-1}$ satisfies conditions
  \ref{F1}--\ref{F4}.
\end{proposition}
\par
As a result, thanks to Theorem~\ref{neutral:Dhat_properties}, we can deduce
results concerning the local existence and uniqueness of solutions of
equation~\eqref{neutral:infdelay} on $\Om\times BC$ which are analogous to those
obtained in Section~\ref{secfun} for
equation~\eqref{neutral:transformed_family}. Therefore, we can deduce the
following result concerning the existence of solutions
of~\eqref{neutral:infdelay} with initial data in $BC$.
\begin{proposition}
  Under
  assumptions~\ref{neutral:D_linear}--\ref{neutral:D_stable},~\ref{N1}--\ref{N3},
  for each $\w\in\Om$ and each $x\in BC$, the
  system~\eqref{neutral:infdelay}$_\w$ locally admits a unique solution
  $z(\cdot,\w,x)$ with initial value $x$, i.e.~$z(s,\w,x)=x(s)$ for each
  $s\in (-\infty,0]$.
\end{proposition}
\begin{proof}
  Since $\wh D(\w,x)\in \Om\times BC$, from
  Proposition~\ref{neutral:transformed_conditions1234} and
  Proposition~\ref{pro:exitBC}, we deduce that
  system~\eqref{neutral:transformed_family}$_\w$ locally admits a unique
  solution $\wh z(\cdot , \wh D(\w,x))$. Hence, taking
  $\wh u(t,\wh D(\w,x))=\wh z_t(\wh D(\w,x))$ and
  $u(t,\w,x)= (\wh D^{-1})_2(\w{\cdot}t, \wh u(t,\wh D(\w,x)))$ for $t\ge 0$ as
  above, we conclude that
  \[
    z(t,\w,x)=\begin{cases}
      \,x(t)&\text{if }t\leq 0,\\
      \, u(t,\w,x)(0)&\text{if }t\geq 0,
    \end{cases}
  \]
  satisfies the statement.
\end{proof}
Analogously, from Proposition~\ref{continuitybolas}, we deduce the continuous
dependence for the product metric topology on sets of the form $\Om\times B_r$
for each $r\ge 0$.
\begin{proposition}\label{continuitybolasN} Under
  assumptions~\ref{neutral:D_linear}--\ref{neutral:D_stable}
  and~\ref{N1}--\ref{N3}, the local map
  \begin{equation*}
    \begin{array}{ccl}
      \mathcal{U}\subset\R^+\times\Om\times B_r& \longrightarrow & \Om\times BC\\
      (t,\w,x) & \mapsto &(\w{\cdot}t,u(t,\w,x))
    \end{array}
  \end{equation*}
  is continuous when we take the restriction of the compact-open topology to
  $B_r$, i.e.~if $t_n\to t$, $\w_n\to\wt\w$ and
  $x_n\stackrel{\textup\di\;}\to \wt x$ as $n\uparrow\infty$ with
  $x_n,\, x\in B_r$ for all $n\in\N$, then $\w_n{\cdot}t_n\to \wt\w{\cdot}t$ and
  $u(t_n,\w_n,x_n)\stackrel{\textup\di\;}\to u(t,\wt \w,\wt x)$ as
  $n\uparrow\infty$.
\end{proposition}
\begin{proof}
  It is an easy consequence of the relation
  \begin{equation}\label{u-uhat}
    \wh D(\w{\cdot}t,u(t,\w,x))=(\w{\cdot}t,\wh u(t,\wh D(\w,x)))\,,
  \end{equation}
  Theorem~\ref{neutral:Dhat_properties} and Proposition~\ref{continuitybolas}.
\end{proof}
As in~Theorem~4.2 of~\cite{paper:obvi}, the following monotonicity theorem,
whose proof is omitted, is an immediate consequence of
Theorem~\ref{preliminaries:monotone} and
Proposition~\ref{neutral:transformed_conditions1234}.
\begin{theorem}\label{neutral:monotone_tau}
  Under assumptions~\ref{neutral:D_linear}--\ref{neutral:D_stable}
  and~\ref{N1}--\ref{N4}, for each $\w\in\Omega$ and $x$, $y\in BC$ such that
  $(\w,x)\leq_{D,\,A} (\w,y)$, it holds that
  \[\tau(t,\w,x)\leq_{D,\,A} \tau(t,\w,y)\] for all $t\ge 0$ where they are
  defined.
\end{theorem}
\begin{lemma}\label{neutral:bounds_D} Under
  assumptions~\ref{neutral:D_linear}--\ref{neutral:D_stable}, there exist
  positive constants $K_D$ and $K_D'$ such that
  \[
    K_D=\sup_{\w\in\Om}\n{D(\w,{\cdot})}=\sup_{\w\in\Om}\n{\wh
      D_2(\w,{\cdot})}\quad \text{and} \quad K_D'=\sup_{\w\in\Om}\n{(\wh
      D^{-1})_2(\w,{\cdot})}\,.
  \]
\end{lemma}
\begin{proof}
  The map $\Om\to\R$, $\w\to\n{D(\w,{\cdot})}$ is continuous thanks to
  \ref{neutral:D_linear}. Since $\Om$ is compact, there is a $K_D>0$ such that
  $K_D=\sup_{\w\in\Om}\n{D(\w,{\cdot})}$. Fix $(\w,x)\in\Om\times B_1$; then
  \begin{align*}
    \n{\wh D_2(\w,x)}_\infty & =\sup_{s\leq 0}\n{D(\w{\cdot}s,x_s)}\leq
                               K_D\,\n{x_s}_\infty\leq K_D \quad  \text{and}\\
    \n{D(\w, x)}& =\n{\wh D_2(\w,x)(0)}\leq\n{\wh
                  D_2(\w,x)}_\infty\,,
  \end{align*}
  whence $K_D=\sup_{\w\in\Om}\n{\wh D_2(\w,{\cdot})}$. As for the bound for
  $\wh D^{-1}$, it follows immediately from Proposition~\ref{prop:bounds}(ii).
\end{proof}
The next result provides the main properties of the trajectory and the
omega-limit set of a bounded solution.
\begin{proposition}
  Assume~\ref{neutral:D_linear}--\ref{neutral:D_stable} and~\ref{N1}--\ref{N3}.
  If $z({\cdot},\w_0,x_0)$ is a solution of~\eqref{neutral:infdelay}$_{\w_0}$
  bounded for the norm $\n{{\cdot}}_\infty$, then the set
  $\{u(t,\w_0,x_0)\mid t\ge 0\}$ is relatively compact when the compact-open
  topology is considered on $BC$ and the omega-limit set of the trajectory of
  the point $(\w_0,x_0)\in\Om\times BC$, defined as
  \[\mathcal{O}(\w_0,x_0)=\{(\w,x)\mid \exists \,t_n\uparrow
    \infty\;\textup{ with } \,\w_0{\cdot}t_n\to\w\,,\;
    u(t_n,\w_0,x_0)\stackrel{\di\;}\to x\}\] is a nonempty, compact and
  invariant subset of $\Om\times BU$ admitting a flow extension.
\end{proposition}
\begin{proof} From~\eqref{u-uhat}, we deduce that
  $\wh u(t,\wh D(\w_0,x_0))=\wh D_2(\w_0{\cdot}t,u(t,\w_0,x_0))$, whence
  \begin{equation}\label{OOhat}
    \mathcal{O}(\w_0,x_0)=\wh D^{-1}  \big(\wh{\mathcal{O}}(\w_0,\wh D_2(\w_0,x_0))\big)\,,
  \end{equation}
  where $\wh{\mathcal{O}}(\w_0,\wh D_2(\w_0,x_0))$ denotes the omega-limit set
  corresponding to the transformed
  system~\eqref{neutral:transformed_family}$_{\w_0}$ with initial datum
  $\wh D_2(\w_0,x_0)$.  From the boundedness of $u(t,\w_0,x_0)$ and
  Lemma~\ref{neutral:bounds_D}, we deduce the boundedness of
  $\wh u(t,\wh D(\w_0,x_0))$ and, consequently,
  Propositions~\ref{neutral:transformed_conditions1234}, \ref{trajectory}
  and~\ref{omegalimit} imply that $\{\wh u(t,\wh D(\w_0,x_0))\mid t\ge 0\}$ is a
  relatively compact subset of $BC$ and
  $\wh{\mathcal{O}}(\w_0,\wh D_2(\w_0,x_0))$ is a compact and invariant subset
  of $\Om\times BU$ admitting a flow extension.  Finally,~\eqref{OOhat} and
  Theorem~\ref{neutral:Dhat_properties}(iii) imply that
  $\{u(t,\w_0,x_0)\mid t\ge 0\}$ is a relatively compact subset of $BC$ and
  $\mathcal{O}(\w_0,x_0)$ is a compact subset of $\Om\times BU$. The invariance
  and flow extension follow the proof of Proposition~\ref{omegalimit}, taking
  into account that, now, Proposition~\ref{continuitybolasN} holds.
\end{proof}
In order to obtain the 1-covering property of some omega-limit sets, in addition
to Hypotheses~\ref{N1}--\ref{N4}, the uniform stability and the componentwise
separating property are assumed.
\begin{enumerate}[resume*=neutral_properties]\setlength\itemsep{0.4em}
\item\label{N5} If $(\w,x)$, $(\w,y)\in\Om\times BC$ admit a backward orbit
  extension, $(\w,x)\leq_{D,\,A} (\w,y)$, and there is a subset
  $J\subset\{1,\ldots,m\}$ such that
  \begin{equation}\label{conditionsN6}
    \begin{split}
      \wh D_2(\w,x)_i=\wh D_2(\w,y)_i & \quad \text{ for each } i\notin J\,,\\
      \wh D_2(\w,x)_i(s)< \wh D_2(\w,y)_i(s) & \quad \text{ for each } i\in
      J\;\text{ and } s\leq 0\,,
    \end{split}
  \end{equation}
  then $G_i(\w,y)-G_i(\w,x) - [A\,(D(\w,y)-D(\w,x))]_i> 0$ for each $i\in J$.
\item\label{N6} For each $k\in\N$, $B_k$ is uniformly stable for the order
  $\leq_{D,\,A}$.\smallskip
\end{enumerate}
\par\smallskip
As in Proposition~\ref{neutral:transformed_conditions1234}, these two
assumptions provide~\ref{F5} and~\ref{F6} for $F=G\circ\wh D^{-1}$.
\begin{proposition}\label{F5F6}
  Under assumptions~\ref{neutral:D_linear}--\ref{neutral:D_stable} and
  \ref{N1}--\ref{N4}, if $G$ satisfies~\ref{N5} and~\ref{N6}, then the map
  $F=G\circ\wh D^{-1}$ satisfies \ref{F5} and \ref{F6}.
\end{proposition}
\begin{proof} First, we check~\ref{F6}. Notice that
  $\wh u(t,\w,x)=\wh D_2(\w{\cdot}t,u(t,\wh D^{-1}(\w,x)))$.  Fix a $k>0$. From
  Lemma~\ref{neutral:bounds_D}, there is a $\wt k>0$ such that, if $x\in B_k$,
  then $(\wh D^{-1})_2(\w,x)\in B_{\wt k}$ for each $\w\in\Om$. In addition, it
  is not hard to check that
  \begin{equation}\label{BkBkt}
    \begin{array}{c}
      x, y\in B_k \\[.1cm]
      x\leq_A y \;\text{ or }\; y\leq _A x
    \end{array}\Longrightarrow
    \begin{array}{l} \wt x= (\wh D^{-1})_2(\w,x),\; \wt y= (\wh D^{-1})_2(\w,y)\in B_{\wt k} \\[.1cm]
      (\w,\wt x) \leq _{D,\,A} (\w,\wt y)
      \;\text{ or }\; (\w,\wt y)\leq _{D,\,A} (\w,\wt x)
    \end{array}\,.
  \end{equation}
  Next, let $r=\max(2\,\wt k,1)$ and $\eps>0$. From
  Theorem~\ref{neutral:Dhat_properties}, there is a $\delta_1>0$ such that
  \begin{align}\label{delta11}
    \{x\in BC\mid \di(x,0)<\delta_1\}\subset \{x\in BC \mid \n{x(0)}\leq 1\} \;& \text{ and}\\ \label{delta12}
    \di(\wh D_2(\w,x),\wh D_2(\w,y))< \eps \quad \forall\, \w\in\Om \text{ and }  x,y\in B_r \text{ with } &\di(x,y)<\delta_1\,.
  \end{align}
  Now, from assumption~\ref{N6}, given this $\delta_1>0$, there is a
  $\delta_2>0$ such that if $\wt x$, $\wt y\in B_{\wt k}$ satisfy
  $\di(\wt x,\wt y)<\delta_2$ and $(\w,\wt x)\leq_{D,\,A} (\w,\wt y)$ or
  $(\w,\wt y)\leq_{D,\,A} (\w,\wt x)$, then
  $\di(u(t,\w,\wt x),u(t,\w,\wt y))< \delta_1$ for each $t\geq 0$.\par
  Moreover, with the notation of~\eqref{BkBkt}, again
  Theorem~\ref{neutral:Dhat_properties} provides a $\delta>0$ such that, for
  each $\w\in\Om$ and $x$, $y\in B_k$ with $\di(x,y)<\delta$, it holds
  \begin{equation}\label{cotaD-1}
    \di(\wt x,\wt y)=\di\big((\wh D^{-1})_2(\w,x),(\wh D^{-1})_2(\w,y)\big)< \delta_2\,.
  \end{equation}
  Altogether, if $x$, $y\in B_k$ satisfy $\di(x,y)<\delta$ and $x\leq_A y$ or
  $y\leq_A x$, from~\eqref{BkBkt} and~\eqref{cotaD-1}, we deduce that
  $(\w,\wt x) \leq _{D,\,A} (\w,\wt y)$ or $(\w,\wt y)\leq _{D,\,A} (\w,\wt x)$
  and $\di(\wt x,\wt y)<\delta_2$. Then, as seen above,
  $\di(u(t,\w,\wt x),u(t,\w,\wt y))< \delta_1$ for each $t\geq 0$ and, from this
  and~\eqref{delta11}, we get $\n{(u(t,\w,\wt x)-u(t,\w,\wt y))(0)}\leq 1$ for
  each $t\geq 0$, which together with $\wt x$, $\wt y\in B_{\wt k}$ yields
  $u(t,\w,\wt x)-u(t,\w,\wt y)\in B_r$. Finally, from~\eqref{delta12} and the
  linear character of $\wh D_2$, we conclude that
  \[
    \di(\wh D_2(\w{\cdot}t,u(t,\w,\wt x),\wh D_2(\w{\cdot}t,u(t,\w,\wt
    y)))=\di(\wh u(t,\w,x),\wh u(t,\w,y)) <\eps
  \]
  for each $t\geq 0 $ and~\ref{F6} holds, as claimed. The verification
  that~\ref{F5} is fulfilled is easier and it is omitted.
\end{proof}
As a consequence, the asymptotic behavior of bounded trajectories  with initial datum $x_0$ such that $\wh D_2(\w_0,x_0)$ satisfies~\ref{functional:regularity}, reproduces exactly the dynamics exhibited by the time variation of the equation, as claimed.
\begin{theorem}\label{copybaseR}
  Let $(\w_0,x_0)\in \Om\times BC$. Under
  assumptions~\ref{neutral:D_linear}--\ref{neutral:D_stable}
  and~\ref{N1}--\ref{N6}, if $\wh D_2(\w_0,x_0)$ satisfies
  property~\ref{functional:regularity} and $z(\cdot,\w_0,x_0)$ is a solution
  of~\eqref{neutral:infdelay}$_{\w_0}$ bounded for the norm
  $\n{{\cdot}}_\infty$, then the omega-limit set
  $\mathcal{O}(\w_0,x_0)=\{(\w,c(\w))\mid \w\in\Om\}$ is a copy of the base and
  \[\lim_{t\to\infty}
    \textup\di(u(t,\w_0,x_0),c(\w_0{\cdot}t))=0\,,\] where $c:\Om\to BU$ is a
  continuous equilibrium, i.e. $u(t,\w,c(\w))=c(\w{\cdot}t)$ for each $\w\in\Om$
  and $t\geq 0$, and it is continuous for the compact-open topology on $BU$.
\end{theorem}
\begin{proof} From Propositions~\ref{neutral:transformed_conditions1234}
  and~\ref{F5F6}, $F=G\circ\wh D^{-1}$ and the corresponding family of
  systems~\eqref{neutral:transformed_family} satisfies
  assumptions~\ref{F1}--\ref{F6}. In addition, as above,
  $\wh u(t, \w_0,\wh D_2(\w_0,x_0))=\wh D_2(\w_0{\cdot}t,u(t,\w_0,x_0))$ and
  Lemma~\ref{neutral:bounds_D} yield the boundedness of
  $\wh z(\cdot, \w_0,\wh D_2(\w_0,x_0))$, the solution
  of~\eqref{neutral:transformed_family}$_{\w_0}$. As a consequence, from
  Theorem~\ref{copiabasenuevo}, we deduce that the omega-limit
  $\wh{\mathcal{O}}\big(\w_0,\wh D_2(\w_0,x_0)\big)$ is a copy of the base, that
  is,
  \[\wh{\mathcal{O}}\big(\w_0,\wh
    D_2(\w_0,x_0)\big)=\{(\w,\wh c(\w))\mid \w\in\Om\}\,,\] where
  $\wh c:\Om\to BU$ is a continuous equilibrium and
  \[\lim_{t\to\infty} \textup\di\big(\wh
    u(t,\w_0,\wh D_2(\w_0,x_0)),\wh c(\w_0{\cdot}t)\big)=0\,.\] Finally,
  from~\eqref{OOhat}, we have
  $\mathcal{O}(\w_0,x_0)=\wh D^{-1} \wh{\mathcal{O}}(\w_0,\wh D_2(\w_0,x_0))$,
  and we conclude that $\mathcal{O}(\w_0,x_0)=\{(\w,c(\w))\mid \w\in\Om\}$ with
  $c(\w)=(\wh D^{-1})_2(\w,\wh c (\w))$ or, equivalently,
  $\wh D^{-1}(\w,\wh c(\w))=(\w,c(\w))$, from which the proof is easily
  finished.
\end{proof}
\noindent{\bf Acknowledgment.}
The authors would like to thank the two anonymous referees for their careful reading of the manuscript and their valuable comments.


\begin{thebibliography}{99}
\bibitem{paper:arbo} {\sc O. Arino, F. Bourad}, {\itshape On the asymptotic
    behavior of the solutions of a class of scalar neutral equations generating
    a monotone semiflow}, J. Differential Equations \textbf{87} (1990), 84--95.
\bibitem{paper:arha} {\sc O. Arino, E. Haourigui}, {\itshape On the asymptotic
    behavior of solutions of some delay differential systems which have a first
    integral}, J. Math. Anal. Appl.  \textbf{122} (1987), 36--46.
\bibitem{book:cohn} {\sc D.L. Cohn} \textit{Measure Theory}, Second Edition,
  Birkh\"{a}user/Springer, New York, 2013.
\bibitem{paper:driver} {\sc R.D. Driver}, Existence and stability of solutions
  of a delay-differential system, \emph{Arch. Rational Mech. Anal.} \textbf{10}
  (1962), 401--426.
\bibitem{book:Ellis} {\sc R.~Ellis}, {\em Lectures on Topological Dynamics},
  Benjamin, New York, 1969.
\bibitem{paper:gyor} {\sc I. Gy\H{o}ri}, {\itshape Connections between
    compartmental systems with pipes and integro-differential equations}, Math.
  Modelling {\bf 7} (1986), 1215--1238.
\bibitem{HKW1990} {\sc J.R. Haddock, T. Krisztin, J. Wu},
        Asymptotic equivalence of neutral and infinite retarded
        differential equations, {\em Nonlinear Anal.} \textbf{14}
        No. 4 (1990), 369-377.
\bibitem{paper:gyel} {\sc I. Gy\H{o}ri, J. Eller}, {\itshape Compartmental
    systems with pipes}, Math. Biosci. {\bf 53} (1981), 223--247.
\bibitem{paper:gywu} {\sc I. Gy\H{o}ri, J. Wu}, {\itshape A neutral equation
    arising from compartmental systems with pipes}, J. Dynam.  Differential
  Equations \textbf{3} No.2 (1991), 289--311.
\bibitem{book:hale} {\sc J.K. Hale}, \textit{Theory of Functional Differential
    Equations}, Applied Mathematical Sciences {\bf 3}, Springer-Verlag, Berlin,
  Heidelberg, New York 1977.
\bibitem{book:hale2} {\sc J.K. Hale, S.M. Verduyn Lunel}, \textit{Introduction
    to Functional Differential Equations}, Applied Mathematical Sciences {\bf
    99}, Springer-Verlag, Berlin, Heidelberg, New York 1993.
\bibitem{book:hino} {\sc Y. Hino, S. Murakami, T. Naito}, \textit{Functional
    Differential Equations with Infinite Delay}, Lecture Notes in
  Math. \textbf{1473}, Springer-Verlag, Berlin, Heidelberg, 1991.
\bibitem{paper:jacq} {\sc J.A. Jacquez}, {\itshape Compartmental Analysis in
    Biology and Medicine}, Third Edition, Thomson-Shore Inc., Ann Arbor,
  Michigan, 1996.
\bibitem{paper:jasi} {\sc J.A. Jacquez, C.P. Simon}, {\itshape Qualitative
    theory of compartmental systems}, SIAM Review {\bfseries 35} No. 1 (1993),
  43--79.
\bibitem{paper:jizh} {\sc J. Jiang, X.-Q. Zhao}, {\itshape Convergence in
    monotone and uniformly stable skew-product semiflows with applications},
  J. Reine Angew. Math {\bfseries 589} (2005), 21--55.
\bibitem{paper:krwu} {\sc T. Krisztin, J. Wu}, {\itshape Asymptotic Periodicity,
    Monotonicity, and Oscillation of Solutions of Scalar Neutral Functional
    Differential Equations}, J. Math.  Anal. Appl. {\bfseries 199} (1996),
  502--525.
\bibitem{paper:MNO} {\sc V. Mu\~{n}oz-Villarragut, S. Novo, R. Obaya}, Neutral
  functional differential equations with applicactions to compartmental systems,
  \emph{SIAM J. Math. Anal.} \textbf{40} No. 3 (2008), 1003--1028.
\bibitem{paper:NOS2007} {\sc S. Novo, R. Obaya, A.M. Sanz}, Stability and
  extensibility results for abstract skew-product semiflows, {\em
    J. Differential Equations\/} \textbf{235} No. 2 (2007), 623--646.
\bibitem{paper:NOV} {\sc S. Novo, R. Obaya, V.M. Villarragut}, Exponential
  ordering for nonautonomous neutral functional differential equations, {\em
    SIAM J. Math. Anal.\/} \textbf{41} (2009), 1025--1053.
\bibitem{paper:obvi} {\sc R. Obaya, V.M. Villarragut}, Exponential ordering for
  neutral functional differential equations with non-autonomous linear
  $D$-operator, \emph{J. Dynam. Differential Equations} \textbf{23} (2011),
  695--725.
\bibitem{paper:obvi2} {\sc R. Obaya, V.M. Villarragut}, Direct exponential
  ordering for neutral compartmental systems with non-autonomous
  $D$-operator. {\em Discrete Contin. Dyn. Syst. Ser. B} \textbf{18} No. 1
  (2013), 185--207.
\bibitem{paper:sawano} {\sc K. Sawano}, Some considerations on the fundamental
  theorems for functional differential equations with infinite delay,
  \emph{Funkcialaj Ekvacioj} \textbf{25} (1982), 97--104.
\bibitem{paper:seifert} {\sc G. Seifert} Positively invariant closed sets for
  systems of delay differential equations, \emph{J. Differential Equations}
  \textbf{22} (1976), 292--304.
\bibitem{book:shyi} {\sc W. Shen, Y. Yi}, Almost Automorphic and Almost Periodic
  Dynamics in Skew-Product Semiflows, {\em Mem. Amer. Math. Soc.}  {\bf 647},
  Amer. Math. Soc., Providence 1998.
\bibitem{paper:smth2} {\sc H.L. Smith, H.R. Thieme}, Strongly order preserving
  semiflows generated by functional differential equations, {\em J. Differential
    Equations\/} \textbf{93} (1991), 332--363.
\bibitem{staf83} {\sc O.J. Staffans}, A Neutral FDE with Stable D-operator is
        Retarded,  {\em J. Differential Equations\/}  \textbf{49}
        (1983), 208--217.
\bibitem{tesis:victor} {\sc V. Mu\~{n}oz-Villarragut}, A dynamical theory for
  monotone neutral functional differential equations with application to
  compartmental systems, Ph.Dissertation, 2010, Universidad de Valladolid.
  https://dialnet.unirioja.es/servlet/tesis?codigo=22569.
\bibitem{paper:WW1985} {\sc Z. Wang, J. Wu}, Neutral Functional Differential
  Equations with Infinite delay, {\em Funkcial.  Ekvac.\/} \textbf{28} (1985),
  157--170.
\bibitem{paper:wazh} {\sc Y. Wang, X.Q. Zhao}, Global convergence in monotone
  and uniformly stable recurrent skew-product semiflows. {\em Infinite
    dimensional dynamical systems}, Fields Inst. Commun., \textbf{64} (2013),
  391--406.
\bibitem{paper:jwu1991} {\sc J. Wu}, Unified treatment of local theory of NFDEs
  with infinite delay, {\em Tamkang J. Math.\/} \textbf{22} No. 1 (1991),
  51--72.
\bibitem{paper:wufr} {\sc J. Wu, H.I. Freedman}, Monotone semiflows generated by
  neutral functional differential equations with application to compartmental
  systems, \emph{Can. J. Math.}  \textbf{43} (1991), 1098--1120.
\bibitem{paper:WZ2002} {\sc J. Wu, X.-Q. Zhao}, Diffusive monotonicity and threshold dynamics of delayed reaction diffusion equations {\em J. Differential Equations\/} \textbf{186} (2002), 470--484.
\end{thebibliography}
\end{document}